\documentclass[pdf, 10pt]{amsart}

\usepackage{amsmath,amstext,amssymb,amsopn,amsthm,mathrsfs,wasysym,dsfont,comment,xcolor}
\usepackage[OT1]{fontenc}
\usepackage{comment}
\usepackage{bbm}

\usepackage{tikz}
\usetikzlibrary{arrows}
\usepackage{float}
\allowdisplaybreaks

\DeclareMathOperator{\dist}{dist}

\newtheorem{thm}{Theorem}[section]
\newtheorem{cor}[thm]{Corollary}

\newtheorem{prop}[thm]{Proposition}

\newtheorem{lem}[thm]{Lemma}

\newtheorem{conj}[thm]{Conjecture}

\def\R{{\mathbb{R}}}
\def\T{{\mathbb{T}}}
\def\C{{\mathbb{C}}}
\def\N{{\mathbb{N}}}
\def\Z{{\mathbb{Z}}}
\def\S{{\mathbb{S}}}

\newcommand{\Pc}{{\mathcal{P}}}\newcommand{\Hc}{{\mathcal{H}}}\newcommand{\Ic}{{\mathcal{I}}}\newcommand{\Ec}{{\mathcal{E}}}

\renewcommand{\b}{\beta}
\newcommand{\n}{{\textbf{n}}}\newcommand{\m}{{\textbf{m}}}\newcommand{\Gc}{{\mathcal G}}

\renewcommand{\o}{\omega}

 \numberwithin{equation}{section}

\begin{document}

\thanks{The first author is partially supported by the  NSF grant DMS-1800305.
	The second author is supported by the National Science Centre of Poland within the research project OPUS 2017/27/B/ST1/01623.
}

\subjclass[2010]{Primary 42A45, Secondary 11L05}
\keywords{Weyl sums, moment curve, paraboloid}
\begin{abstract}
We prove moment inequalities for exponential sums with respect to singular measures, whose Fourier decay matches those of curved hypersurfaces. Our emphasis will be on proving estimates that are sharp with respect to the scale parameter $N$, apart from $N^\epsilon$ losses. In a few instances, we manage to remove these losses.

\end{abstract}

\title[] {Restriction of exponential sums to hypersurfaces}

\author{Ciprian Demeter}
\address{Department of Mathematics, Indiana University,  Bloomington IN}
\email{demeterc@indiana.edu}

\author{Bartosz Langowski}
\address{Department of Mathematics, Indiana University,  Bloomington IN \newline
Wroc\l{}aw University of Science and Technology,\newline
Faculty of Pure and Applied Mathematics, Wroc\l{}aw, Poland}
\email{balango@iu.edu}

\maketitle

\section{Description of the questions}

Let $e(t)=e^{2\pi i t}$. Given a 1-periodic $d$-dimensional exponential sum, what can be said about its restriction to a given smooth manifold $\mathcal{M}$ in $\T^d$? In this paper we investigate a few examples, by restricting attention to the case when $\mathcal{M}$ is a curved hypersurface.
\smallskip

Our most substantial findings will concern the exponential sums along the moment curve.
For $d\ge 2$, $N\in\N$, sequences $a=(a_n)\in\C$ and  $x=(x_1,\ldots,x_d)\in \T^d$ (identified with $[0,1]^d$), we let
\begin{align*}
S_{a, d}(x, N)=\sum_{n=1}^N a_n\, e(x_1 n+ \dots+ x_d n^d).
\end{align*}
 In the special case $a_n\equiv 1$ we shall simplify the notation to $S_{d}(x, N)$. The investigation of their $L^p$ moments on the measure space $(\T^d,dx)$  has been at the forefront of both harmonic analysis and analytic number theory for many years, culminating in the full resolution of this problem, see \cite{BDG} and \cite{Wol}. The estimate
 \begin{equation}\label{a26}
 \int_{\T^d} |S_{a, d}(x, N)|^p\,dx\lesssim_\epsilon\|a\|_{\ell^2}^p\begin{cases}N^\epsilon, \;&0<p\le d(d+1)\\ N^{\frac{p}{2}-\frac{d(d+1)}{2}+\epsilon},\;&p>d(d+1)\end{cases}
 \end{equation}
 is known to be sharp, up to the $N^\epsilon$ term. Here and everywhere else $\epsilon$ denotes a positive, arbitrarily small constant.

  Our interest here lies in the behavior of such sums when they are restricted to hypersurfaces in $\T^d$. For example, speaking somewhat informally, can the ``large" values of $S_{a, d}(x, N)$ ``concentrate" on such a singular set as a hypersurface? Such questions are already interesting and difficult for the constant sequence $a_n\equiv 1$. But as we shall soon see, the arbitrary coefficient case comes with additional motivation.
\smallskip

We make the following conjecture. It predicts that in a certain range of $L^p$ spaces,  the behavior of exponential sum $S_{a,d}(x,N)$ restricted to  $\mathcal{M}$ is governed by square root cancellation.
\begin{conj}\label{conj:1}
	Let $\sigma$ be the surface measure of a smooth hypersurface $\mathcal{M}$ in $\T^d$ with non--vanishing Gaussian curvature. Then for each  $N\ge 1$, $d\ge 2$ and for each sequence $a$, the estimate
	\begin{equation}\label{eq:conj1}
	\int_{\T^d} |S_{a, d}(x, N)|^p\,d\sigma(x)\lesssim_\epsilon \|a\|_{\ell^2}^p\begin{cases}N^\epsilon,&p\le p_d:=d(d-1)\\
	N^{\frac{p-p_d}{2}+\epsilon},&p> p_d\end{cases}
	\end{equation}
	holds for all $\epsilon>0$, with the implicit constant depending only on $\sigma$ and $d$.
\end{conj}
Note that the estimate \eqref{eq:conj1} at the critical exponent   $p_d$ implies the estimates for all other values of $p$. For $p<p_d$ this follows from H\"older, while for $p>p_d$, by interpolation with the trivial bound at $p=\infty$. Throughout this paper, interpolation will always refer to combining two estimates via H\"older's inequality.

To see that -apart from the $N^\epsilon$ term-  \eqref{eq:conj1} is optimal, notice that given the box
$$
Q_{d, N}:=[0,1/N]\times [0,1/N^2]\times\ldots\times [0,1/N^{d-1}]\times [0,1/N^{d}]
$$
we can always find a smooth hypersurface
$$\mathcal{M}=\{(x,F(x)):\;x\in [0,1]^{d-1}\}\subset [0,1]^d$$ with nonzero Gaussian curvature, and satisfying
\begin{equation}\label{eq:fit}
F\left([0,1/N]\times [0,1/N^2]\times\ldots\times [0,1/N^{d-1}]\right)\subset [0,1/N^{d}].
\end{equation}
Indeed, we may take $F(x_1)=x_1^2$ when $d=2$, $F(x_1,x_2)=x_1x_2$ when $d=3$, $F(x_1,x_2,x_3)=\frac{x_2^2+x_1x_3}{2}$ when $d=4$, and  $F(x_1,x_2,x_3,x_4)=\frac{x_1x_4+x_2x_3}{2}$ when $d=5$. In general, we may take $$F(x_1,\ldots,x_{d-1})=\frac2{d}\sum_{1\le i\le d/2}x_ix_{d-i}.$$

Given $\mathcal{M}$ satisfying \eqref{eq:fit}, we let $\sigma$ be its surface measure. Taking the constant sequence $a_n\equiv1$, we have $\|a\|_{\ell^2}=N^{1/2}$. On the other hand, we have  constructive interference, $|S_{d}(x, N)|\simeq N$ for $x\in Q_{d,N}$, so we can estimate
\begin{align*}
\int_{\T^d} |S_{d}(x, N)|^p\,d\sigma(x)\ge \int_{Q_{d,N}} |S_{d}(x, N)|^p\,d\sigma(x)\simeq N^p \sigma(Q_{d,N})\simeq N^p\, N^{-d(d-1)/2}.
\end{align*}
Thus \eqref{eq:conj1} is sharp if $p>p_d$. By considering random sequences, we also see the optimality for $p\le p_d$.
\smallskip

It is not clear whether the requirement that the hypersurface has nonzero Gaussian curvature is needed in order for \eqref{eq:conj1} to hold. In fact, this inequality continues to hold if $\mathcal{M}$ is the graph of any measurable $F$ satisfying $\|F\|_{L^\infty([0,1]^{d-1})}=O(1/N^d)$. In this case, summation by parts and the fact that the variation norm of the sequence $e(n^dF(x))$ is $O(1)$, reduces \eqref{eq:conj1} to an application of \eqref{a26} with $d$ replaced with $d-1$. Since the critical exponent for Conjecture \ref{conj:1} coincides with the critical exponent for \eqref{a26} in dimension $d-1$, it is tempting to fantasize on whether there is a clever way to quickly derive \eqref{eq:conj1} from \eqref{a26} for arbitrary, merely smooth  $F$. We were not able to prove this, and we believe such a direct argument is unlikely to exist.

\smallskip

One might suspect that the resolution of Conjecture \ref{conj:1} would rely on  suitable use of the asymptotic formula
due to Herz \cite{Herz}
\begin{equation}\label{eq:herz}
\widehat{\sigma}(\xi) = C\left(\frac{\xi}{|\xi|}\right)|\xi|^{-\frac{d-1}{2}} \cos\left( 2\pi \left( \sup_{x\in \mathcal{M}} (x\cdot \xi)-\frac{d-1}{8}\right)\right) + \mathcal{O}(|\xi|^{-\frac{d+1}{2}}),
\end{equation}
where $C$ is some positive continuous function.

One can ask what happens if we assume merely the relevant decay of the Fourier transform of the measure, ignoring the oscillations in the formula \eqref{eq:herz}. Clearly, one should expect a weaker assertion. We propose the following conjecture, that will provide us with the main line of attack for Conjecture  \ref{conj:1}.
\begin{conj}\label{conj:2}
	Let $d\ge 2$.  Let $\sigma$ be a positive finite measure on $\T^d$ satisfying the Fourier transform bound
	\begin{equation}\label{eq:decbeta}
	|\widehat{\sigma}(\xi)|\lesssim (1+|\xi|)^{-\beta},\;\xi\in \Z^d,
	\end{equation}
	with $\beta=(d-1)/2$. Let
	$$
	\rho_d:=\begin{cases}\frac{3d^2-4}{4},\;d\text{ even}\\\frac{3d^2-3}{4},\;d\text{ odd}\end{cases}.
	$$
	We have
	\begin{equation}\label{eq:conj2}
	\int_{\T^d} |S_{a, d}(x, N)|^p\,d\sigma(x)\lesssim_\epsilon  \|a\|_{\ell^2}^p\begin{cases}N^\epsilon,&0<p\le \rho_d\\N^{\frac{p-\rho_d}{2}+\epsilon},&p>\rho_d\end{cases}.
	\end{equation}	
\end{conj}

Note that once \eqref{eq:conj2} is established for some $p$, the inequality for smaller values of $p$ follows using H\"older. Observe also that $p_2=\rho_2=2$, $p_3=\rho_3=6$, but $\rho_d<p_d$ for $d\ge 4$. We add one more conjecture, that will provide us with the strategy to approach  Conjecture  \ref{conj:2}.

\begin{conj}\label{conj:3}
	Let $d\ge 2$. For each $j\ge 0$ the following estimate holds
	\begin{align}\label{eq:conj3}
	2^{ j \frac{d+1}{2}}\int_{[0,2^{-j}]^d} \left|{S}_{a, d}(x, N)\right|^p \, dx\lesssim_\epsilon
	\|a\|_{\ell^2}^p\begin{cases}N^\epsilon,&0<p\le \rho_d\\N^{\frac{p-\rho_d}{2}+\epsilon},&p>\rho_d\end{cases}.
	\end{align}
\end{conj}
Since the sequence $a$ is arbitrary, the domain of integration $[0,2^{-j}]^d$ in \eqref{eq:conj3} may be replaced with any of its translates in $\T^d$.

Conjectures \ref{conj:2} and \ref{conj:3} are related in two ways. On the one hand, if \eqref{eq:conj2} holds for some $p$, then \eqref{eq:conj3} will also hold for the same $p$. In particular, Conjecture \ref{conj:2} implies Conjecture \ref{conj:3}. On the other hand, the validity of \eqref{eq:conj3} for some even integer $p$ will be seen to imply the validity of \eqref{eq:conj2} for the same $p$ (and thus, also for all smaller $p$). So Conjecture \ref{conj:3} implies Conjecture \ref{conj:2} whenever $d$ is not divisible by 4. See Proposition \ref{a27} for details.

Note that \eqref{eq:conj3} for some $p\le \rho_d$ implies the same inequality for all smaller exponents, via H\"older. In fact, the exponent  $\frac{d+1}{2}$ of $2^{j}$
can be taken to be larger for smaller values of $p$, but that will not concern us. Given this exponent, the power of $N$ in Conjecture \ref{conj:3}, and thus also in Conjecture \ref{conj:2}, is sharp. Indeed, let us test \eqref{eq:conj3} with the constant sequence $a_n\equiv 1$ and  $2^{j}=N^k$,  where $k=\frac{d+1}{2}$ if $d$ is odd, and $k=\frac{d+2}{2}$ if $d$ is even. Since
$$|S_{d}(x,N)|\simeq N,\;\;\text{for } |x_1|,\ldots,|x_{k-1}|\ll 2^{-j}\text{ and }|x_k|\ll N^{-k},\ldots,|x_d|\ll N^{-d},$$
we find that
$$2^{ j \frac{d+1}{2}}\int_{[0,2^{-j}]^d} \left|{S}_{d}(x, N)\right|^pdx\gtrsim N^{k\frac{d+1}{2}}N^{-k(k-1)}N^{\frac{(k-1)k}{2}-\frac{d(d+1)}{2}}N^{p}=\|a\|_{\ell^2}^pN^{\frac{p-\rho_d}{2}}.$$
When $p\le \rho_d$, \eqref{eq:conj3} is seen to be sharp for $j=0$, by testing with randomized sequences.

We caution that we have stated Conjecture \ref{conj:3} assuming that the worst case scenario when $p>\rho_d$ is provided by constructive interference near a point. It is possible  that new obstructions arise from more sophisticated interferences. However, we will prove that this is not the case in dimensions $d\le 5$.

To put Conjecture \ref{conj:3} into perspective, we compare it with the following ``folklore" conjecture, a comprehensive generalization of \eqref{a26} (see for example Section 13.6 in \cite{C}).
\begin{conj}
	\label{M4}
Let $\beta_1<\beta_2<\ldots<\beta_k$ be  positive integers. For each $p\ge 2$ and $a_n\in\C$	we have 
	
	$$\|\sum_{n=1}^N a_ne(x_1n^{\beta_1}+\ldots+x_kn^{\beta_k})\|_{L^{p}([0,1]^k)}
	\lesssim_{\epsilon}N^{\epsilon}(1+N^{\frac12-\frac{\beta_1+\ldots+\beta_k}{p}})\|a_n\|_{l^2}.$$
\end{conj}
When $d$ is odd, \eqref{eq:conj3} follows from Conjecture \ref{M4} with $\beta_i=\frac{d+1}2+i$, $1\le i\le k=\frac{d-1}{2}$. Indeed, it suffices to note that $\rho_d=2(\beta_1+\ldots+\beta_k)$, and to write
$$\int_{[0,2^{-j}]^d} \left|{S}_{a, d}(x, N)\right|^p \, dx\le 2^{-j\frac{d+1}{2}}\sup_{b:\;|b_n|=|a_n|}\int_{[0,1]^{\frac{d-1}{2}}}|\sum_{n=1}^N b_ne(x_{\beta_1}n^{\beta_1}+\ldots+x_{d}n^{\beta_k})|^pdx.$$
However, sharp results (at the critical exponent) for such incomplete systems are out of reach at the moment, even in the simplest case $k=2$, $\beta_1=1$, $\beta_2=3$. We follow a rather different approach in this paper, that takes advantage of the oscillatory effect coming from the first $\frac{d+1}{2}$ variables.  
\medskip

Let us now describe another case of interest. It involves the exponential sums with frequencies supported on dilates of the unit sphere $\S^{d-1}$
$$S^{\S}_{a,d}(x,N)=\sum_{\n\in\sqrt{N}\S^{d-1}\cap \Z^d}a_{\n}e(\n\cdot x).$$
In \cite{BR}, the following conjecture is made.
\begin{conj}
\label{a39}Let $\mathcal{M}$ be a real analytic hypersurface in $\T^d$, with nonnegative curvature and surface measure $\sigma$. Then
\begin{equation}
\label{a40}
\|S^{\S}_{a,d}(x,N)\|_{L^2(d\sigma)}\lesssim \|a\|_{\ell^2}.
\end{equation}
\end{conj}
Note that the implicit constant in \eqref{a40} is scale-independent.
As $S^{\S}_{a,d}(x,N)$ is the typical eigenfunction with eigenvalue $-2\pi N$ for the Laplacian on $\T^d$, this conjecture is part of a vast literature, that we do not recall here, but rather point the reader to \cite{BR} for references. Inequality \eqref{a40} is verified in \cite{BR} for $d=2$ and $d=3$. Proving \eqref{a40} for $d=2$, as well as  with an $N^\epsilon$ bound for $d=3$ is rather easy, and it follows by using only the Fourier decay \eqref{eq:decbeta}.
The bulk of the paper \cite{BR} is devoted to proving the scale-independent bound for $d=3$.  The argument is a wonderful blend of analysis and number theory of all flavors. Most crucially, it relies on the oscillatory nature of $\widehat{d\sigma}$, as expressed by \eqref{eq:herz}. The authors prove
\begin{equation}
\label{a41}
\sum_{\n_1,\n_2\in \sqrt{N}\S^{d-1}\cap \Z^d}a_{\n_1}\overline{a_{\n_2}}\;\widehat{d\sigma}(\n_1-\n_2)\lesssim \|a\|_{\ell^2}^2
\end{equation}
by exploiting subtle cancellations between the Fourier coefficients of $d\sigma$.

While some of the ingredients needed to extend this argument to $d=4$ are known (e.g. \cite{BR} contains lots of Jarnik-type estimates, while the energy estimate was proved in \cite{BD2}), pursuing this remains challenging due to the delicate nature of the oscillatory component of $\widehat{d\sigma}$.

We note that if we insist on proving scale-independent inequalities, we cannot use  $L^p$ with $p>2$ in place of $L^2$ in \eqref{a40}, when $d>2$. This follows by invoking constructive interference, as before. However, there are some interesting questions left open in two dimensions, regarding larger values of $p$. See Section  \ref{sec:sphere}, where we establish a connection with a conjecture of Cilleruelo and Granville.
\smallskip

We investigate the analogous questions for the paraboloid
$$\mathbb{P}^{d-1}_N=\{\n=(n_1,\ldots,n_{d-1},n_1^2+\ldots+n_{d-1}^2)\in\Z^d:\;1\le n_i\le N \},$$
by considering the sums
$$S_{a,d}^{\mathbb{P}}(x,N)=\sum_{\n\in \mathbb{P}^{d-1}_N}a_{\n}e(\n\cdot x), \;x\in\T^d.$$
The motivation for investigating the paraboloid is twofold. First, it has positive principal curvatures, like the sphere. But the distribution of the lattice points on the two manifolds is very different, and the term $\widehat{d\sigma}(\n_1-\n_2)$ in \eqref{a41} is sensitive to that. This makes the question for the paraboloid of independent interest. An additional motivation is provided by the special significance of the exponential sums $S_{a,d}^{\mathbb{P}}(x,N)$: they are solutions to the free Schr\"odinger equation on $\T^d$.
\smallskip

We hesitantly make the following conjecture.
\begin{conj}
\label{a42}
Let $\mathcal{M}$ be a real analytic hypersurface in $\T^d$, with nonnegative curvature and surface measure $\sigma$.
For each $d\ge 2$, and $N\ge 1$ we have the estimate
$$\|S^{\mathbb{P}}_{a,d}(x,N)\|_{L^2(d\sigma)}\lesssim \|a\|_{\ell^2}.$$
\end{conj}

\subsection*{Notation}
\vspace{0.3in}

For a measurable (or finite) set $A$ in $\R^n$ (or $\Z^n$) we denote by $|A|$ its measure (or cardinality, respectively).
We will use either the notation $X \lesssim Y$ or $X=O(Y)$
to indicate that $|X| \le CY$, with a positive constant $C$ independent of variable parameters such as scales (usually denoted by $N,M$), sequences and functions.
We shall write $X \simeq Y$ when simultaneously $X \lesssim Y$ and $Y \lesssim X$. Finally, for positive $Y$ we will write either $X\ll Y$ or $X=o(Y)$ if  $|X|\le cY$ holds with some  small enough positive constant $c$, independent of variable parameters.

\section{Main results and methodology}

Given a finitely supported function $f:\Z^d\to\C$, its Fourier transform is
$$\widehat{f}(x)=\sum_{\n\in\Z^d}f(\n)e(\n\cdot x),\;\;x\in\T^d.$$
Given a finite complex measure $\nu$ on $\T^d$ ($d|\nu|(\T^d)<\infty$), its Fourier coefficients are
$$\widehat{d\nu}(\n)=\int_{\T^d}e(-\n\cdot x)d\nu(x).$$
The following lemma is classical. Its origins can be traced at least as far back as the work of Hardy and Littlewood \cite{HL}. The argument uses the fact that $p$ is even, and this is not just an artifact of the proof. The lemma is false for $p=3$ (see \cite{GR}) and in fact for any real number $p>2$ that is not an even integer (see Theorem 3.2 in \cite{MS}). The counterexamples use $d\mu(x)=d\nu(x)=dx$.
\begin{lem}
	\label{a28}
Let $\mu,\nu$ be finite  measures on $\T^d$. Assume   $\nu$ has positive Fourier coefficients and   satisfies $|\widehat{d\mu}(\n)|\le \widehat{d\nu}(\n)$ for all $\n\in\Z^d$.
Then for each finitely supported $f:\Z^d\to\C$ and for each  positive even  integer $p$ we have
$$|\int_{\T^d}|\widehat{f}|^{p}d\mu|\le \int_{\T^d}|\widehat{g}|^{p}d\nu,$$
where $g=|f|$.	
\end{lem}
\begin{proof}
We use Plancherel's formula first
\begin{align*}
\int_{\T^d}|\widehat{f}|^{p}d\mu=\sum_{\n\in\Z^d}\overline{\widehat{|\widehat{f}|^{p}}(\n)}\;\widehat{d\mu}(\n).
\end{align*}
An easy computation shows that, since $p$ is even,
$$\left|\widehat{|\widehat{f}|^{p}}(\n)\right|\le \widehat{|\widehat{g}|^{p}}(\n).$$
Thus
$$|\int_{\T^d}|\widehat{f}|^{p}d\mu|\le \sum_{\n\in\Z^d}\overline{\widehat{|\widehat{g}|^{p}}(\n)}\;\widehat{d\nu}(\n)=\int_{\T^d}|\widehat{g}|^{p}d\nu.$$
\end{proof}

The lemma allows us to prove the connection between Conjecture \ref{conj:2} and Conjecture \ref{conj:3}.

\begin{prop}
\label{a27}
If \eqref{eq:conj2} holds for some $p$, then \eqref{eq:conj3} also holds for the same $p$.  If \eqref{eq:conj3} holds  for some positive even integer $p$, then \eqref{eq:conj2} also holds for the same $p$.
\end{prop}
\begin{proof}
For the first part, we find it more convenient to identify $\T^d$ with $[-1/2,1/2]^d$.
Let $\eta:\R^d\to[0,\infty]$ be a smooth function satisfying
$$1_{[-1/4,1/4]^d}\le \eta\le 1_{[-1/2,1/2]^d}.$$
The Fourier coefficient  $\widehat{\eta}(\n)$ coincides with the Fourier transform of $\eta$ evaluated at $\n$.
 Then $d\sigma(x)={2^{j\frac{d+1}{2}}}\eta(2^{j-2}x)dx$ is a finite positive measure on $\T^d$ that satisfies \eqref{eq:decbeta} uniformly over $j$. Note also that $1_{[0,2^{-j}]^4}(x)\le \eta(2^{j-2}x)$.

For the second part, let $\eta:\R^d\to [0,\infty)$ be a smooth nonnegative function, compactly supported in $[-1/2,1/2]^d$, with positive Fourier transform satisfying $\widehat{\eta}(\xi)\gtrsim 1_{[-1,1]^d}(\xi)$. The positive measure $\nu$ on $\T^d$ given by
$$d\nu(x)=\sum_{j:\;1\le 2^j\le dN^d}2^{j\frac{d+1}{2}}\eta(2^jx)dx$$
is finite and satisfies
$$|\widehat{d\sigma}(\n)|\lesssim \widehat{d\nu}(\n),\;|\n|\le dN^d$$
if
$$|\widehat{d\sigma}(\n)|\lesssim (1+|\n|)^{-\frac{d-1}{2}}.$$
Let
$$\phi(x)=\sum_{\n\in\Z^d:\;|\n|\le dN^d}\widehat{d\sigma}(\n)e(\n\cdot x).$$
We now apply Lemma \ref{a28} to the pair $(d\mu=\phi(x)dx,d\nu)$ and to $\widehat{f}(x)=S_{a, d}(x, N)$
\begin{align}
\label{a36}
\nonumber\int_{\T^d} |S_{a, d}(x, N)|^p\,d\sigma(x)&=\int_{\T^d} |S_{a, d}(x, N)|^pd\mu(x)\\\nonumber&\le \int_{\T^d} |S_{|a|, d}(x, N)|^pd\nu(x)\\&\lesssim \sum_{j:\;1\le 2^j\le dN^d}2^{j\frac{d+1}{2}}\int_{[-2^{-j},2^{-j}]^d}|S_{|a|, d}(x, N)|^pdx.
\end{align}
It is now immediate that \eqref{eq:conj3} implies \eqref{eq:conj2}.

\end{proof}

We will use inequality \eqref{a35} stated below when $|I||lS-lS|\ll 1$.
\begin{lem}
\label{a37}	
Let $(a_n)_{n\in S}$ be a finite sequence of complex numbers, and let $p=2l$ be a positive even integer. Write
$$lS-lS=\{n_1+\ldots+n_l-n_{l+1}-\ldots-n_{2l}:\;n_i\in S\}.$$
Then
$$|\sum_{n\in S}a_n|^p\le |lS-lS|\int_{0}^1|\sum_{n\in S}a_ne(nt)|^pdt.$$
In particular, for each $I\subset [0,1]$ we have
\begin{equation}
\label{a35}
\int_I|\sum_{n\in S}a_ne(nt)|^pdt\le |I||lS-lS|\int_{0}^1|\sum_{n\in S}a_ne(nt)|^pdt.
\end{equation}
\end{lem}
\begin{proof}
It suffices to observe that
$$|\sum_{n\in S}a_n|^p=\int_{0}^1|\sum_{n\in S}a_ne(nt)|^p\sum_{m\in lS-lS}e(mt)dt.$$
\end{proof}
 In our paper we only investigate the subcritical regime in Conjectures \ref{conj:1}, \ref{conj:2} and \ref{conj:3}. Combining \eqref{a26} at $p=d(d+1)$ with H\"older gives, for each $p<d(d+1)$
\begin{equation}
\label{a38}
2^{ j \frac{d+1}{2}}\int_{[0,2^{-j}]^d} \left|{S}_{a, d}(x, N)\right|^p \, dx\lesssim_\epsilon N^\epsilon 2^{j(\frac{1-d}{2}+\frac{p}{d+1})}\|a\|_{\ell^2}^{p}.
\end{equation}
This shows that Conjecture \ref{conj:3} holds for $p\le \frac{d^2-1}{2}$, and thus Conjectures \ref{conj:1} and  \ref{conj:2} also hold for $p\le \frac{d^2-1}{2}$ when $d$ is odd, and for $p\le \frac{d^2-4}{2}$ when $d$ is even. This is a rather poor range of exponents, and we will improve it significantly in low dimensions. Let us now state our main results.

\begin{thm}
\label{a30}
Conjecture \ref{conj:3} holds in the full range for $d=2$ and $d=3$, and in the range $p\le 10$ for $d=4$.

Moreover, when  $d=4$, we have the following superficially weaker, but morally equivalent substitute for \eqref{eq:conj3}, in the full range $p\le \rho_4=11$
\begin{align}\label{eq:conj3d=4}
2^{ \frac{5}{2}j}\int_{[0,2^{-j}]^4} \left|{S}_{a, 4}(x, N)\right|^p \, dx\lesssim_\epsilon N^{p(\frac12-\frac13)+\epsilon}\|a\|_{\ell^6}^p.
\end{align}

When $d=5$, the following holds in the full range $p\le \rho_5=18 $.
\begin{align}\label{eq:conj3d=5}
 2^{3j}\int_{[0,2^{-j}]^5} \left|{S}_{a, 5}(x, N)\right|^p \, dx\lesssim_\epsilon N^{p(\frac12-\frac19)+\epsilon}\|a\|_{\ell^9}^p.
 \end{align}
In particular,  \eqref{eq:conj3} holds in the full range when $d\le 5$, for constant coefficients $a_n\equiv 1$.
\end{thm}
The proof for $d=2$ involves elementary methods. The case $d=3$ as well as the range $p\le 10$ for $d=4$ will rely  on known decouplings. To reach $p=11$ for $d=4$ and $p=18$ for $d=5$   we need to develop new small cap decoupling technology. The inspiration comes from a result proved in \cite{bourg1}, and streamlined in \cite{C2}. Our Theorem \ref{a13} extends the result from \cite{bourg1} to the case when small cap is applicable to three (rather than two) of the variables, by removing the periodicity assumption on the second variable. Its linear counterpart, Theorem \ref{a31}, will be used in multiple forms, see Corollaries \ref{a10}, \ref{a11}, \ref{a32} and \ref{d32}. In Section \ref{d=5last} we extend this approach to five dimensions.

The systematic study of small cap decoupling at the critical exponent for the moment curve has been initiated only recently, and it mostly addresses dimensions $2\le d\le 4$. See \cite{DGW}, \cite{C2} and \cite{H}.
While partial results exist in higher dimensions, see e.g. \cite{Oh}, they are not strong enough to fully solve Conjecture  \ref{conj:3} in dimensions $d\ge 6$.
\smallskip

The following result is a direct consequence of Theorem \ref{a30} and Proposition \ref{a27}.
\begin{cor}
Conjectures \ref{conj:1} and \ref{conj:2} hold true in the full range for $d=2$ and $d=3$, and in the range $p\le 10$ for $d=4$.
When $d=5$ and $p\le 18$, we have the following morally equivalent substitute for \eqref{eq:conj2}
$$\int_{\T^d} |S_{a, d}(x, N)|^p\,d\sigma(x)\lesssim_\epsilon  N^{p(\frac12-\frac19)+\epsilon}\|a\|_{\ell^9}^p.$$
In particular, both conjectures are verified in the range $p\le 18$ for the constant sequence $a\equiv 1$.
\end{cor}
\smallskip

Inequality \eqref{eq:conj2} has been proved  in \cite{CS}, see Theorem 2.2 and Example 2.3 there, in the supercritical regime  $p\ge d(d+1)>\rho_d$. The argument relies on a simple application of \eqref{a26} for $p=d(d+1)$, similar to the subcritical estimate \eqref{a38}.

Let us make a quick comparison between our methods and the ones in \cite{CS}. The latter also makes implicit use of Lemma \ref{a28}, by noting that each $\sigma$ satisfying \eqref{eq:decbeta}
with $\beta=(d-1)/2$ also satisfies
$$|\widehat{d\sigma}(\n)1_{F_N}(\n)|\lesssim \widehat{d\nu}(n)=\sum_{1\le 2^j\lesssim N^d}{2^{-j\frac{d-1}{2}}}1_{A_{j}(N)}(\n).$$
Here $A_{j}(N)$ is the intersection of the annulus $|\n|\simeq 2^j$ with the frequency support $$F_N=\{\n=(n_1,\ldots,n_d):\;|n_1|\lesssim N,\ldots,|n_d|\lesssim N^d\}$$ of $S_{a,d}(x,N).$ The argument in \cite{CS} continues by using the rather weak estimate
$$d|\nu|(x)\le \sum_{1\le 2^j\le dN^d}{2^{-j\frac{d-1}{2}}}{|A_{j}(N)|}1_{\T^d}(x)dx,$$
that destroys important cancellations. Since $|A_j(N)|\simeq N^{1+2+\ldots+(k_j-1)}2^{j(d+1-k_j)}$ if $N^{k_j-1}\lesssim 2^j\lesssim N^{k_j}$, $1\le k_j\le d$,
the application of Lemma \ref{a28} gives for each positive even integer $p$
$$\int_{\T^d} |S_{a, d}(x, N)|^p\,d\sigma(x)\lesssim \sum_{1\le 2^j\le dN^d}N^{\frac{(k_j-1)k_j}2}2^{j(\frac{d+3}{2}-k_j)}\int_{\T^d} |S_{|a|, d}(x, N)|^pdx.$$
We may use Lemma \ref{a37} to show that this upper bound is larger than the one in  \eqref{a36}. Indeed, for each $j$ we apply \eqref{a35} for each of the first $k_j-1$ variables
$$
2^{j\frac{d+1}{2}}\int_{[-2^{-j},2^{-j}]^d}|S_{|a|, d}(x, N)|^pdx\le
N^{\frac{(k_j-1)k_j}2}2^{j(\frac{d+3}{2}-k_j)}\int_{\T^{k_j-1}\times [0,2^{-j}]^{d+1-k_j}} |S_{|a|, d}(x, N)|^pdx.$$
The smallness of the range of the last variables will be used crucially in our argument, by means of exploiting the small cap decoupling  phenomenon.
\smallskip

In most cases, we have formulated our conjectures with the $N^\epsilon$ term in the bound. While we do this mostly out of abundance of caution, it must be pointed out that one of our favorite tools, Theorem \ref{a2}, is known to have a genuine logarithmic loss in the scale. In any case, we do manage to prove a few scale-independent results. One of them is the sharp form of  \eqref{eq:conj1} for $p=4$ and $d=3$, proved in Section \ref{four} using  counting arguments. We also prove the scale-independent version of Conjecture \ref{conj:1} for $d=2$, in the full range. In fact, in Section \ref{sec:L2paraboloid} we prove sharp $L^2$ estimates for the paraboloids in all dimensions, subject only to the decay of $\widehat{d\sigma}$. The bounds are scale-independent only in the case $d=2$. When $d=3$, there is a logarithmic loss. This is very similar to the situation  for the sphere described in the previous section.  The question of removing this logarithmic loss by exploiting the oscillations of $\widehat{d\sigma}$ remains open.

\section{Decouplings and exponential sum estimates, old and new}
\label{review}

In this section we record the main tools we will use to address Conjecture \ref{conj:3}.

\begin{thm}[$L^2$ orthogonality]
\label{a1}	
Let $(\xi_k)_{1\le k\le K }$ be a $1/R$-separated sequence of real numbers. Then, for each interval $\omega$ of length $\gtrsim R$ and for each $a_k\in\C$ we have
$$\|\sum_{k=1}^Ka_ke(\xi_kx)\|_{L^2(\omega)}\lesssim |\omega|^{1/2}\|a_k\|_{\ell^2}.$$ 	
\end{thm}

\begin{thm}[$l^2(L^6)$ decoupling for curved planar curves, \cite{BD}]
\label{a2}	
Let $I\subset [0,1]$ be an interval. Assume $\phi_1,\phi_2$ are $C^{3}$ real-valued curves on the interval $I$ satisfying the curvature condition
$$\left|\det\begin{bmatrix}\phi_1'(t)&\phi_2'(t)\\\\
\phi_1^{''}(t)&\phi_2^{''}(t)&\end{bmatrix}\right|\simeq 1,\;t\in I.$$
Partition $I$ into intervals $H$ or length $1/\sqrt{R}$. Then for each  collection of points $\xi_k\in I$,  each $a_k\in\C$ and for each square $Q$ in $\R^2$ with diameter $\gtrsim R$
we have
$$\|\sum_{k=1}^Ka_ke(x_1\phi_1(\xi_k)+x_2\phi_2(\xi_k))\|_{L^6(Q)}\lesssim_{\epsilon}R^\epsilon(\sum_{H}\|\sum_{\xi_k\in H}a_ke(x_1\phi_1(\xi_k)+x_2\phi_2(\xi_k))\|_{L^6(Q)}^2)^{1/2}.$$
\end{thm}

We now introduce our new decouplings. Let us start with the case of curves in four dimensions.
Let
$\phi_2,\phi_3,\phi_4$ will be real analytic functions defined on some open interval containing $[\frac12,1]$, and  satisfying
\begin{equation}
\label{1}
\|\phi_k'\|_{C^3([\frac12,1])}=\sum_{1\le n\le 4}\max_{\frac12\le t\le 1}|\phi_k^{(n)}(t)|\lesssim 1, \;\;k\in\{2,3,4\},
\end{equation}
the Wronskian condition
\begin{equation}
\label{2}|W(1,\phi_2',\phi_3',\phi_4')|=
\left|\det
\begin{bmatrix}\phi_2^{(2)}(t)&\phi_2^{(3)}(t)&\phi_2^{(4)}(t)\\
\phi_3^{(2)}(t)&\phi_3^{(3)}(t)&\phi_3^{(4)}(t)\\ \phi_4^{(2)}(t)& \phi_4^{(3)}(t)&\phi_4^{(4)}(t)
\end{bmatrix}\right|\simeq 1,\;\;t\in[\frac12,1],
\end{equation}
and
\begin{equation}
\label{3}
\left|\det\begin{bmatrix}\phi_2''(t)&\phi_3''(t)\\\phi_2''(s)&\phi_3''(s)\end{bmatrix}\right|\simeq 1,\;\; t,s\in [\frac12,1] \text{ with }|t-s|\simeq 1.
\end{equation}

\begin{thm}[Bilinear small cap $l^2L^{12}$  decoupling]
\label{a13}Let $I_1,I_2$ be intervals of length $\simeq
 N$ in $[\frac{N}{2},N]$, with $\dist(I_1,I_2)\simeq {N}$.	
Let $\Omega=[0,1]\times\omega_2\times\omega_3\times\omega_4$, where $\omega_i$ are intervals satisfying $|\omega_2|, |\omega_3|\ge N^2$, $|\omega_4|\ge N$. Then we have
$$\int_{\Omega}|\prod_{j=1}^2\sum_{n\in I_j}a_ne(nx_1+\phi_2(\frac{n}{N})x_2+\phi_3(\frac{n}{N})x_3+\phi_4(\frac{n}{N})x_4)|^{6}dx\lesssim_\epsilon N^\epsilon|\Omega|\|a\|_{\ell^2}^{12}.$$ 	
\end{thm}
The case $\phi_2(t)=t^2$ was explained in \cite{C2}. In this context, the requirements \eqref{2} and \eqref{3} are superficially weaker than, but essentially equivalent (cf. Exercise 7.10 in \cite{C}) to those stated in \cite{C2}, which we recall below
$$
\left|\det
\begin{bmatrix}
\phi_3^{(3)}(t)&\phi_3^{(4)}(t)\\ \phi_4^{(3)}(s)&\phi_4^{(4)}(s)
\end{bmatrix}\right|\simeq 1,\;\;t,s\in[\frac12,1],
$$
and
$$
|\phi_3^{(3)}(t)|\simeq 1,\;\; t\in [\frac12,1].
$$
The novelty of this formulation is that it allows small cap in three of the four variables.
The argument will crucially exploit periodicity in the $x_1$ variable, so it does not accommodate a small cap on this component. Our main new observation is that periodicity on the second component is not needed.
\smallskip

It is easy to see that the linear version of Theorem \ref{a13} is not true. Let us first observe that if $p>2$ then
$$\|a\|_{\ell^2([N/2,N])}\lesssim \|a\|_{\ell^p([N/2,N])}N^{\frac12-\frac1p}, $$
and the inequality is an equivalence when $a=1_{[N/2,N]}$.
The $l^2$ norm turns out to be too strong. In fact, we show that for all $p<6$,  the inequality
$$\int_{[-1,1]^2\times[-1/N,1/N]\times[-1/N^3,1/N^3]}|\sum_{N/2\le n\le N}a_ne(nx_1+\phi_2(\frac{n}{N})N^2x_2+\phi_3(\frac{n}{N})N^3x_3+\phi_4(\frac{n}{N})N^4x_4)|^{12}dx$$$$\lesssim_\epsilon N^{\epsilon-4}\|a\|_{\ell^p}^{12}N^{6-\frac{12}{p}}$$
is false for the sequence $a=1_{[\frac{N}{2},\frac{N}{2}+M]}$, with $M=N^{3/4}$. Consider the set $S$ of points
\begin{equation}
\label{a25}
(x_1,\ldots,x_4)\in[-1,1]\times[-1,1]\times [-1/N,1/N]\times[-1/N^3,1/N^3]\end{equation}
satisfying
$$\begin{bmatrix}1&\phi_2'(\frac{1}{2})&\phi_3'(\frac{1}{2})&\phi_4'(\frac{1}{2})\\\\0&\phi_2''(\frac{1}{2})&\phi_3''(\frac{1}{2})&\phi_4''(\frac{1}{2})\\\\0&\phi_2'''(\frac{1}{2})&\phi_3'''(\frac{1}{2})&\phi_4'''(\frac{1}{2})\\\\0&\phi_2''''(\frac{1}{2})&\phi_3''''(\frac{1}{2})&\phi_4''''(\frac{1}{2})\end{bmatrix}\begin{bmatrix}x_1\\\\Nx_2\\\\N^2x_3\\\\N^3x_4\end{bmatrix}=\begin{bmatrix}o(1/M)\\\\o(N/M^2)\\\\o(N^2/M^3)\\\\o(N^3/M^4)\end{bmatrix}.$$
Since the entries $o(1/M),o(N/M^2),o(N^2/M^3),o(N^3/M^4)$ are all $o(1)$, using \eqref{1} and \eqref{2}, it follows that any solution to this system satisfies
$$|x_1|,|Nx_2|,|N^2x_3|,|N^3x_4|\le 1. $$
In particular, \eqref{a25} is guaranteed to hold. Also, for each $2\le k\le 4$ and $m\le M$ we have $$\sup_{t\in[1/2,1]}|\phi_k^{(5)}(t)(\frac{m}{N})^5N^{k}x_k|\ll 1.$$
We use Taylor's formula  with fifth order remainder, to write for each $k$ and $n=N/2+m$
$$\phi_k(\frac{n}{N})=\sum_{l=0}^4\phi^{(l)}_k(\frac12)(\frac{m}{N})^l+\phi^{(5)}_k(t_{k,m})(\frac{m}{N})^5.$$
Combining all these facts, we see that we have constructive interference
$$|\sum_{N/2\le n\le N/2+M}e(nx_1+\phi_2(\frac{n}{N})N^2x_2+\phi_3(\frac{n}{N})N^3x_3+\phi_4(\frac{n}{N})N^4x_4)|\simeq M$$
for each $x\in S$. Since $|S|\simeq 1/M^{10}$, it follows that if $a=1_{[N/2,N/2+M]}$
$$\int_{[-1,1]^2\times[-1/N,1/N]\times[-1/N^3,1/N^3]}|\sum_{N/2\le n\le N}a_ne(nx_1+\phi_2(\frac{n}{N})N^2x_2+\phi_3(\frac{n}{N})N^3x_3+\phi_4(\frac{n}{N})N^4x_4)|^{12}dx$$
$$\gtrsim M^{2}.$$
This lower bound is significantly bigger than $N^{\epsilon-4}\|a\|_{\ell^p}^{12}N^{6-\frac{12}{p}}$, when $p<6$.
\medskip

However, the next result shows that we can work with $p=6$. We will need to add another determinant condition, that is  satisfied in all our applications
\smallskip

\begin{equation}
\label{4}
\left|\det\begin{bmatrix}\phi_2''(t)&\phi_3''(t)\\\phi_2''(t)&\phi_3''(t)\end{bmatrix}\right|\simeq 1,\;\; t\in [\frac12,1].
\end{equation}
\smallskip

\begin{thm}[Linear small cap $l^6L^{12}$ decoupling]
	\label{a31}	
	Assume $\phi_2,\phi_3,\phi_4:(0,3)\to\R$ are real analytic and satisfy \eqref{1}, \eqref{2}, \eqref{3} and \eqref{4} on $[1/4,1]$.
	Let $\Omega=[0,1]\times\omega_2\times\omega_3\times\omega_4$, where $\omega_i$ are intervals satisfying $|\omega_2|, |\omega_3|\ge N^2$, $|\omega_4|\ge N$. Then we have
	$$\int_{\Omega}|\sum_{N/2\le n\le N}a_ne(nx_1+\phi_2(\frac{n}{N})x_2+\phi_3(\frac{n}{N})x_3+\phi_4(\frac{n}{N})x_4)|^{12}dx\lesssim_\epsilon (N^{1/3+\epsilon}\|a\|_{\ell^6})^{12}|\Omega|.$$ 	
\end{thm}

 Note that this result is sharp in three ways. First, the factor $N^4$ on the right cannot be made any smaller. This can be seen by using a random sequence $a_n\in\{-1,1\}$. Second, as observed above, the term $N^{1/3}\|a\|_{\ell^6}$ cannot be replaced with the smaller term $N^{1/2-1/p}\|a\|_{\ell^p}$, for $p<6$. Third, none of the intervals $\omega_i$ can be allowed to be significantly smaller than the specified lower bounds. This can be seen by using the constant sequence $a\equiv 1$, which leads to constructive interference on $[0,o(1/N)]\times [0,o(1)]^3$.
\smallskip

We list as corollaries  four particular cases of interest for us. The first one was proved in \cite{bourg1} for the constant sequence $a\equiv 1$. It corresponds to $\phi_2(t)=t^2$, $\phi_3(t)=t^4$, $\phi_4(t)=t^3$.

\begin{cor}
\label{a10}
Let $\bar{\omega}_3,\bar{\omega}_4$ be intervals of length greater than $1/N^2$. Then 	
$$\int_{[0,1]\times[0,1]\times \bar{\omega}_3\times\bar{\omega}_4}|\sum_{N/2\le n\le N}a_ne(nx_1+n^2x_2+n^3x_3+n^4x_4)|^{12}dx\lesssim_\epsilon N^{4+\epsilon}|\bar{\omega}_3||\bar{\omega}_4|\|a\|_{\ell^6}^{12}.$$	
\end{cor}
The second one is a new result, even for the constant sequence. It corresponds to $\phi_2(t)=t^3$, $\phi_3(t)=t^4$, $\phi_4(t)=t^2$. We will see that, at least for our purposes, this is a stronger estimate than the first corollary. That comes from the fact that $\bar{\omega}_2$ is allowed to be much smaller than the periodicity interval $[0,1]$. By renaming the variables, the result is as follows.
\begin{cor}
	\label{a11}
	Let $\bar{\omega}_2,\bar{\omega}_3$ be intervals of length greater than $1/N$. Let $\bar{\omega}_4$ be an interval of length greater than $1/N^2$.
	Then 	
	$$\int_{[0,1]\times\bar{\omega}_2\times \bar{\omega}_3\times\bar{\omega}_4}|\sum_{N/2\le n\le N}a_ne(nx_1+n^2x_2+n^3x_3+n^4x_4)|^{12}dx\lesssim_\epsilon N^{4+\epsilon}|\bar{\omega}_2||\bar{\omega}_3||\bar{\omega}_4|\|a\|_{\ell^6}^{12}.$$	
\end{cor}
Here are two more corollaries that we will use to address the five dimensional moment curve in Section \ref{d=5last}. They will serve as lower dimensional decoupling in the ``easier" regimes.
\begin{cor}
\label{a32}
Let $\bar{\omega}_3,\bar{\omega}_4$ be intervals of length greater than $1/N^2$. Let $\bar{\omega}_5$ be an interval of length greater than $1/N^3$.
Then 	
$$\int_{[0,1]\times\bar{\omega}_3\times \bar{\omega}_4\times\bar{\omega}_5}|\sum_{N/2\le n\le N}a_ne(nx_1+n^3x_3+n^4x_4+n^5x_5)|^{12}dx\lesssim_\epsilon N^{4+\epsilon}|\bar{\omega}_3||\bar{\omega}_4||\bar{\omega}_5|\|a\|_{\ell^6}^{12}.$$
\end{cor}
\begin{cor}
	\label{d32}
	Let $\bar{\omega}_4,\bar{\omega}_5$ be intervals of length greater than $1/N^3$.
	Then 	
	$$\int_{[0,1]^2\times \bar{\omega}_4\times\bar{\omega}_5}|\sum_{N/2\le n\le N}a_ne(nx_1+n^2x_2+n^4x_4+n^5x_5)|^{12}dx\lesssim_\epsilon N^{4+\epsilon}|\bar{\omega}_4||\bar{\omega}_5|\|a\|_{\ell^6}^{12}.$$
\end{cor}

In Section \ref{d=5last} we will pursue a similar investigation in five dimensions. To keep things simple, we confine ourselves to only  proving the following result, that suffices for our applications to the ``hard" regimes. While the technology involved in proving the previous results is of quadratic nature (the governing tool is the decoupling for the parabola), the next theorem will take into consideration the oscillatory effect of cubic terms, similar to the decoupling inequality for the three dimensional moment curve. The small cap decoupling nature of the next result is captured by the variables $x_3$, $x_4$ and $x_5$. The novelty compared to the result in \cite{Oh} is that we allow $x_3$ to have a range as small as $1/N^2$, as opposed to the periodicity range $[0,1]$. This will be crucial to our argument in Section \ref{thenewd=5}.
\begin{thm}
\label{c7}
Let $\bar{\omega}_3,\bar{\omega}_5$ be intervals of length greater than $1/N^2$. Let $\bar{\omega}_4$ be an interval of length greater than $1/N$. Then we have
$$\int_{[0,1]^2\times\bar{\omega}_3\times \bar{\omega}_4\times\bar{\omega}_5}|\sum_{N/2\le n\le N}a_ne(nx_1+n^2x_2+n^3x_3+n^4x_4+n^5x_5)|^{18}dx\lesssim_\epsilon N^{18(\frac12-\frac1{9})+\epsilon}|\bar{\omega}_3||\bar{\omega}_4||\bar{\omega}_5|\|a\|_{\ell^{9}}^{18}.$$
\end{thm}
The exponent $7$ of $N$ on the right hand side is sharp. It seems plausible that the result remains valid with the box $\bar{\omega}_3\times \bar{\omega}_4\times \bar{\omega}_5$ replaced with a smaller one, having volume $N^{-6}$. Also, our argument does not use periodicity for $x_2$, which suggests that perhaps there is a more general formulation of this result, in the style of Theorem \ref{a31}. However, the nondegeneracy conditions associated with such a result are expected to be rather complicated. To avoid unnecessary technicalities, we follow a more economical approach in Section \ref{d=5last}.

\section{Proof of Conjecture \ref{conj:3}: $d=3$}
\label{d=3}

Recall that it suffices to prove the conjectured estimate at the critical exponent $\rho_3=6$. Invoking a dyadic decomposition, it suffices to prove that for $j\ge 0$
$$\int_{[0,2^{-j}]^3}|\sum_{n=N/2}^Na_ne(x_1n+x_2n^2+x_3n^3)|^6dx\lesssim_\epsilon 2^{-2j}N^\epsilon \|a_n\|_{\ell^2}^6.$$

Since we are dealing with the moment curve, it is tempting to use the strongest estimate for it, inequality \eqref{a26} for $d=3$
$$\int_{[0,1]^3}|\sum_{n=N/2}^Na_ne(x_1n+x_2n^2+x_3n^3)|^{12}dx\lesssim_\epsilon N^\epsilon\|a\|_{\ell^2}^{12}.$$
However, the reader may check that either combining this  with H\"older, or interpolating it with the $L^2$ bound (available in the nontrivial range $2^j\le N^2$)
$$\int_{[0,2^{-j}]^3}|\sum_{n=N/2}^Na_ne(x_1n+x_2n^2+x_3n^3)|^{2}dx\lesssim 2^{-3j}\|a\|_{\ell^2},$$
does not produce the desired decay  $2^{-2j}$ for $p=6$. It turns out  that what we have to use instead is decoupling for planar curves.

We distinguish three cases. The variable $x_1$ will not play any role, and abusing notation, we write $x=(x_2,x_3)$.
\\
\\
Case 1. Assume $2^j\le N$.    We will prove the superficially stronger estimate
 \begin{equation}
 \label{a4}
 \int_{[0,2^{-j}]^2}|\sum_{n=N/2}^Na_ne(x_2n^2+x_3n^3)|^6dx\lesssim_\epsilon 2^{-j}N^\epsilon \|a_n\|_{\ell^2}^6.\end{equation}
 Working with the last two variables proves to be the most efficient choice, as it leads to spatial domains of largest possible size, and thus to a more efficient decoupling. Via a change of variables and enlarging the range of $x_2$ to $[0,1]$ in order to give ourselves enough room to decouple,  we write
 $$\int_{[0,2^{-j}]^2}|\sum_{n=N/2}^Na_ne(x_2n^2+x_3n^3)|^6dx\le N^{-5}\int_{[0,N^2]\times [0,N^32^{-j}]}|\sum_{n=N/2}^Na_ne(x_2(\frac{n}{N})^2+x_3(\frac{n}{N})^3)|^6dx.$$
 We cover $[0,N^2]\times [0,N^32^{-j}]$ with squares $Q$ with side length $R=N^2$ and apply  Theorem \ref{a2} to the curve $\phi(t)=(t^2,t^3)$, to get the desired estimate. Note that each interval $H$ of length $1/N=1/\sqrt{R}$ contains only one point $\xi_n=n/N$.
\\
\\
Case 2. Assume $N<2^j\le N^2$. We prove \eqref{a4}.  We will combine $l^2(L^6)$ decoupling with $L^2$ orthogonality as follows. First, we enlarge the range for $x_2$ to $[0,N2^{-j}]$. Then we change variables and decouple using Theorem \ref{a2} for the curve $\phi(t)=(t^2,t^3)$, with $R=N^32^{-j}$.
\begin{align*}
\int_{[0,2^{-j}]^2}|\sum_{n=N/2}^Na_ne(x_2n^2+&x_3n^3)|^6dx\le N^{-5}\int_{ [0,N^32^{-j}]^2}|\sum_{n=N/2}^Na_ne(x_2(\frac{n}{N})^2+x_3(\frac{n}{N})^3)|^6dx\\&\lesssim_\epsilon N^{\epsilon-5}[\;\sum_H[\;\int_{ [0,N^32^{-j}]^2}|\sum_{n\in H}a_ne(x_2(\frac{n}{N})^2+x_3(\frac{n}{N})^3)|^6dx]^{1/3}]^3.
\end{align*}
Here the sum is over intervals $H$ of length $M=(2^j/N)^{1/2}$ partitioning $[N/2,N]$. Since $H$ is small, its corresponding arc on the parabola is essentially flat with respect to our spatial domain. In the absence of curvature, our best tool is $L^2$ decoupling.
For each $H$, we write $|\cdot|^6=|\cdot|^4|\cdot|^2$. We estimate the first term pointwise by $M^2\|a\|_{\ell^2(n\in H)}^4$, using the Cauchy--Schwarz inequality. Since $N^32^{-j}\ge N$, we can estimate the integral of the second term applying Theorem \ref{a1} (using orthogonality in any of the $x_2$ or $x_3$ variables)
$$\int_{ [0,N^32^{-j}]^2}|\sum_{n\in H}a_ne(x_2(\frac{n}{N})^2+x_3(\frac{n}{N})^3)|^6dx\lesssim M^2\|a\|_{\ell^2( H)}^4(N^32^{-j})^2\|a\|_{\ell^2(H)}^2.$$
Plugging this into the previous inequality leads to the proof of  \eqref{a4}.
\\
\\
Case 3.  Assume $N^2<2^j$. Ignoring the first and second variables, it suffices to prove that
$$\int_{[0,2^{-j}]}|\sum_{n=N/2}^Na_ne(x_3n^3)|^6dx_3\lesssim  \|a_n\|_{\ell^2}^6.$$
For $2^{j}=N^2$, this follows from $L^2$ orthogonality as in the previous case, while for larger values, the integral gets smaller.

\section{$d=4$: Proof of \eqref{eq:conj3d=4}}
\label{d=4}
It suffices to deal with $p=11$. We also normalize the sequence $a$ such that $N^{1/3}\|a\|_{\ell^6}=1$. Thus, $\|a\|_{\ell^2}\le 1.$ Let us start by giving a measure of the difficulty of the inequality we need to prove

\begin{equation}
\label{a9}
\int_{[0,2^{-j}]^4}|\sum_{n=1}^Na_ne(x_1n+x_2n^2+x_3n^3+x_4n^4)|^{11}dx\lesssim_\epsilon N^{\epsilon}2^{-\frac{5j}{2}}.
\end{equation}
As in the previous section, it is tempting to use the Vinogradov-type estimate (\eqref{a26} with $d=4$)
\begin{equation}
\label{a8}
\int_{[0,1]^4}|\sum_{n=1}^Na_ne(x_1n+x_2n^2+x_3n^3+x_4n^4)|^{20}dx\lesssim_\epsilon N^\epsilon.
\end{equation}
Combining it with H\"older leads to a weaker estimate  than the one we need (see \eqref{a38})
$$\int_{[0,2^{-j}]^4}|\sum_{n=1}^Na_ne(x_1n+x_2n^2+x_3n^3+x_4n^4)|^{11}dx\lesssim_\epsilon N^\epsilon2^{-\frac{9j}5}.$$

When $2^j\le N^2$, we can do better, by interpolating \eqref{a8} with the inequality
\begin{align}
\label{a12}
\int_{[0,2^{-j}]^4}|\sum_{n=1}^Na_ne(x_1n+x_2n^2+x_3n^3+x_4n^4)|^{6}dx&\lesssim_\epsilon N^\epsilon2^{-3j}\max\{2^{-j},\frac1N\}\|a\|_{\ell^2}^6\\\nonumber&\le N^\epsilon2^{-3j}\max\{2^{-j},\frac1N\}.
\end{align}
The estimate \eqref{a12} follows from $l^2(L^6)$ decoupling for the curve $\phi(t)=(t^3,t^4)$ and $L^2$ orthogonality, as in the previous section. The reader may check that this verifies \eqref{a9} in the range $2^{j}\le N^{\frac98}$.

We may also try something else when $N\le 2^j\le N^2$. From now on, we will replace $[1,N]$ with $[N/2,N]$, via the dyadic decomposition used earlier.

Instead of interpolating the $L^6$ estimate with the $L^{20}$ estimate, we interpolate it with the  $L^{12}$ estimate, the latter being derived from Corollary \ref{a10} as follows
\begin{align}
\label{a46}
\nonumber&\int_{[0,2^{-j}]^4}|\sum_{n=N/2}^Na_ne(x_1n+x_2n^2+x_3n^3+x_4n^4)|^{12}dx\\\nonumber&\le N2^{-j}\int_{[0,1]\times[0,2^{-j}]^3}|\sum_{n=N/2}^Na_ne(x_1n+x_2n^2+x_3n^3+x_4n^4)|^{12}dx\\\nonumber&\le N2^{-j}\int_{[0,1]^2\times [0,2^{-j}]^2}|\sum_{n=N/2}^Na_ne(x_1n+x_2n^2+x_3n^3+x_4n^4)|^{12}dx\\&\lesssim_\epsilon N^{1+\epsilon}2^{-3j}.
\end{align}
The first inequality follows from Lemma \ref{a37}, while the last one follows from Corollary \ref{a10}. Combining this with \eqref{a12} gives
$$\int_{[0,2^{-j}]^4}|\sum_{n=N/2}^Na_ne(x_1n+x_2n^2+x_3n^3+x_4n^4)|^{11}dx\lesssim_\epsilon 2^{-3j}N^{\frac23+\epsilon}.$$
This is only as good as \eqref{a9} if $2^j\ge N^{4/3}$.

In summary, our methods so far leave a gap between $9/8$ and $4/3$.
To cover the remaining part of the range we will rely instead on Corollary \ref{a11}.  In fact, this result is so strong that it will give us \eqref{a9} even for $p=12$, when $2^j\le N^2$.
\\
\\
Case 1. $2^{j}\le N$. We enlarge the domain for $x_1$ to $[0,1]$, then apply Corollary \ref{a11}.
\begin{align*}
\int_{[0,2^{-j}]^4}|&\sum_{n=N/2}^Na_ne(x_1n+x_2n^2+x_3n^3+x_4n^4)|^{12}dx\\&\le\int_{[0,1]\times [0,2^{-j}]^3}|\sum_{n=N/2}^Na_ne(x_1n+x_2n^2+x_3n^3+x_4n^4)|^{12}dx \\&\lesssim_\epsilon N^{\epsilon}2^{-3j}.
\end{align*}
This is actually better than what we need (for $p=12$) by a factor of $2^{-j/2}$.
\\
\\
Case 2. $N\le 2^j\le N^2$. Applying Lemma \ref{a37}, we enlarge the range for $x_1$ to $[0,1]$ and gain the factor $N2^{-j}$. We also enlarge  the range for both $x_2$ and $x_3$ to $[0,1/N]$ and then apply Corollary \ref{a11}
\begin{align*}
\int_{[0,2^{-j}]^4}|&\sum_{n=N/2}^Na_ne(x_1n+x_2n^2+x_3n^3+x_4n^4)|^{12}dx\\&\lesssim N2^{-j} \int_{[0,1]\times [0,1/N]^2\times  [0,2^{-j}]}|\sum_{n=N/2}^Na_ne(x_1n+x_2n^2+x_3n^3+x_4n^4)|^{12}dx\\&\lesssim_\epsilon N^{\epsilon-1}2^{-2j}. \end{align*}
Note that the upper bound is more favorable than the one (i.e. \eqref{a46}) obtained via Corollary \ref{a10}. It is also better than what we need (for $p=12$).
\smallskip

Writing $\sigma=\sigma^N_{high}+\sigma^N_{low}$, with $\sigma^N_{low}$ the Fourier restriction of $\sigma$ to frequencies $\le N^2$, we have proved that
$$\int_{\T^4} |S_{a, 4}(x, N)|^{12}\,d\sigma^N_{low}(x)\lesssim_\epsilon N^\epsilon.$$
As observed in the introduction, these methods cannot prove the similar statement for $\sigma^N_{high}$. Indeed, if $a_n\equiv 1$ and $N^2\ll 2^j\ll  N^4$,  then constructive interference near the origin gives the lower bound
$$\int_{[0,2^{-j}]^4}|\sum_{n=N/2}^Ne(x_1n+x_2n^2+x_3n^3+x_4n^4)|^{12}dx\gtrsim N^52^{-2j}\min(1,N^32^{-j}).$$
This is much bigger than $2^{-\frac{5j}{2}}\|a_n\|_{\ell^2}^{12}$. The best that can be proved with our methods is $p=11$.
\\
\\
Case 3. $N^2\le 2^j\le N^3$. In this regime, Corollary \ref{a10} is stronger than Corollary \ref{a11}. We first apply Lemma \ref{a37} for both the first and the second variables, then Corollary \ref{a10}
\begin{align*}
\int_{[0,2^{-j}]^4}|&\sum_{n=N/2}^Na_ne(x_1n+x_2n^2+x_3n^3+x_4n^4)|^{12}dx\\&\lesssim 2^{-2j}N^3\int_{[0,1]^2\times [0,1/N^2]^2}|\sum_{n=N/2}^Na_ne(x_1n+x_2n^2+x_3n^3+x_4n^4)|^{12}dx\\&\lesssim_\epsilon N^{\epsilon-1}2^{-2j}.
\end{align*}
To get a favorable $L^{11}$ estimate we interpolate this with the following $L^6$ bound that we may get by reasoning as in the previous section. The variables $x_1,x_2$ play no role this time.
\begin{align*}
&\int_{[0,2^{-j}]^4}|\sum_{n=N/2}^Na_ne(x_1n+x_2n^2+x_3n^3+x_4n^4)|^{6}dx\\&\lesssim 2^{-2j}N^{-7}\sup_{b:\;|b_n|=|a_n|}\int_{[0,N^32^{-j}]\times[0,N^42^{-j}] }|\sum_{n=N/2}^Nb_ne(x_3(\frac{n}N)^3+x_4(\frac{n}{N})^4)|^{6}dx_3dx_4\\&\lesssim_\epsilon 2^{-2j}N^{\epsilon-7}\sup_{b:\;|b_n|=|a_n|}[\;\sum_H\;[\int_{[0,N^32^{-j}]\times[0,N^42^{-j}] }|\sum_{n\in H}b_ne(x_3(\frac{n}N)^3+x_4(\frac{n}{N})^4)|^{6}dx_3dx_4]^{1/3}]^{3}.
\end{align*}
For each interval $H$ of length $M=(2^j/N)^{1/2}$ we combine $L^2$ decoupling with Cauchy--Schwarz
\begin{align*}
\int_{[0,N^32^{-j}]\times[0,N^32^{-j}] }|&\sum_{n\in H}b_ne(x_3(\frac{n}N)^3+x_4(\frac{n}{N})^4)|^{6}dx\\&\le M^2\|b\|_{\ell^2(H)}^4\int_{[0,N^32^{-j}]\times[0,N^42^{-j}] }|\sum_{n\in H}b_ne(x_3(\frac{n}N)^3+x_4(\frac{n}{N})^4)|^{2}dx_3dx_4\\&\lesssim M^2N^{7}2^{-2j}\|a\|_{\ell^2(H)}^6
.\end{align*}
Combining the inequalities in the last two paragraphs we get
\begin{equation}
\label{a49}
\int_{[0,2^{-j}]^4}|\sum_{n=N/2}^Na_ne(x_1n+x_2n^2+x_3n^3+x_4n^4)|^{6}dx\lesssim_\epsilon \frac{N^{\epsilon}}{2^{3j}N}\|a\|_{\ell^2}^6\le \frac{N^{\epsilon}}{2^{3j}N}.
\end{equation}
Interpolation leads to
$$\int_{[0,2^{-j}]^4}|\sum_{n=N/2}^Na_ne(x_1n+x_2n^2+x_3n^3+x_4n^4)|^{11}dx$$
$$\lesssim_\epsilon N^\epsilon(\frac{1}{2^{3j}N})^{1/6} (\frac{1}{2^{2j}N})^{5/6}=N^\epsilon\frac1{2^{13j/6}N}.
$$
This is smaller than the desired bound $N^\epsilon2^{-5j/2}$ precisely when $2^{j}\le N^3$.
\\
\\
Case 4. If $2^j\ge N^3$, the result follows from $L^2$ orthogonality. The first three variables play no role here.
\begin{align*}
\int_{[0,2^{-j}]^4}|&\sum_{n=N/2}^Na_ne(x_1n+x_2n^2+x_3n^3+x_4n^4)|^{11}dx\\&\le 2^{-3j}\sup_{b:\;|b_n|=|a_n|}\int_{[0,1/N^{3}]}|\sum_{n=N/2}^Nb_ne(x_4n^4)|^{11}dx_4\\&\le 2^{-3j}N^{9/2}\sup_{b:\;|b_n|=|a_n|}\int_{[0,1/N^{3}]}|\sum_{n=N/2}^Nb_ne(x_4n^4)|^{2}dx_4\\&\lesssim 2^{-3j}N^{3/2}\lesssim 2^{-5j/2}.
\end{align*}
\bigskip

\section{Proof of Conjecture \ref{conj:3}: $d=4$, $p\le 10$}

We prove
\begin{equation}
\label{a21}
\int_{[0,2^{-j}]^4}|\sum_{n=1}^Na_ne(x_1n+x_2n^2+x_3n^3+x_4n^4)|^{10}dx\lesssim_\epsilon N^\epsilon2^{-\frac{5j}{2}}\|a\|_{\ell^2}^{10}.
\end{equation}
Recall (see \eqref{a12} and \eqref{a49}) that for $2^j\le N^3$
\begin{equation}
\label{a20}
\int_{[0,2^{-j}]^4}|\sum_{n=1}^Na_ne(x_1n+x_2n^2+x_3n^3+x_4n^4)|^{6}dx\lesssim_\epsilon N^\epsilon\max\{\frac1N,2^{-j}\}2^{-3j}\|a\|_{\ell^2}^{6}.
\end{equation}
When $2^{j}\le N^2$, \eqref{a21} follows by interpolating
$$
\int_{[0,1]^4}|\sum_{n=1}^Na_ne(x_1n+x_2n^2+x_3n^3+x_4n^4)|^{20}dx\lesssim_\epsilon N^\epsilon\|a\|_{\ell^2}^{20}
$$
with \eqref{a20}.

When $N^2\le 2^j\le N^3$, we use again \eqref{a20}
and the fact that $|\cdot|^{10}=|\cdot |^6|\cdot |^4$ to conclude that
$$
\int_{[0,2^{-j}]^4}|\sum_{n=1}^Na_ne(x_1n+x_2n^2+x_3n^3+x_4n^4)|^{10}dx\lesssim_\epsilon \frac{N^\epsilon}{N2^{3j}}\|a\|_{\ell^2}^{6}N^2\|a_n\|_{\ell^2}^4\lesssim_\epsilon {N^\epsilon}2^{-\frac52j}\|a\|_{\ell^2}^{10}.
$$
When $2^j\ge N^3$, it suffices to replace $[0,2^{-j}]$ with $[0,1/N^3]$ for $x_4$, and to use $L^2$ orthogonality for this variable.

\section{$d=5$: Proof of \eqref{eq:conj3d=5}}

\label{thenewd=5}
It suffices to deal with $p=18$. We  normalize the sequence $a$ such that $N^{\frac12-\frac19}\|a\|_{\ell^9}=1$. In particular, $N^{\frac12-\frac16}\|a\|_{\ell^6}\le 1$ and $\|a\|_{\ell^2}\le 1$. We need to prove

$$
\Ic_j:=\int_{[0,2^{-j}]^5}|\sum_{n=N/2}^Na_ne(x_1n+x_2n^2+x_3n^3+x_4n^4+x_5n^5)|^{18}dx\lesssim_\epsilon N^{\epsilon}2^{-3j}.
$$
Case 1. If $2^j\le N$, the inequality follows from Theorem \ref{c7}, using the inclusion $[0,2^{-j}]^5\subset [0,1]^2\times [0,2^{-j}]^3$.
\\
\\
Case 2. If $N\le 2^j\le N^2$, we combine  Lemma \ref{a37} (for $x_1$) and
 Theorem \ref{c7} as follows
 $$\Ic_j\lesssim N2^{-j}\int_{[0,1]^2\times [0,2^{-j}]\times [0,\frac1N]\times [0,2^{-j}]}|\sum_{n=N/2}^Na_ne(x_1n+x_2n^2+x_3n^3+x_4n^4+x_5n^5)|^{18}dx\lesssim_\epsilon N^{\epsilon}2^{-3j}.$$
When $N^2\le 2^j$, Theorem \ref{c7} becomes inefficient. We use instead lower dimensional methods.
\\
\\
Case 3. When $N^2\le 2^j\le N^3$, the variable $x_2$ plays no role. We combine Lemma \ref{a37} (for $x_1$) with  Corollary \ref{a32} and Cauchy--Schwarz ($18=12+6$)
\begin{align*}
\Ic_j&\lesssim N2^{-j}2^{-j}N^3\sup_{b:\;|b_n|=|a_n|}\int_{[0,1]\times [0,\frac1{N^2}]\times [0,\frac1{N^2}]\times [0,2^{-j}]}|\sum_{n=N/2}^Nb_ne(x_1n+x_3n^3+x_4n^4+x_5n^5)|^{12}dx\\&\lesssim_\epsilon N^{\epsilon}2^{-3j}.
\end{align*}
\\
\\
Case 4. When $N^3\le 2^j\le N^4$, the variable $x_3$ plays no role. We combine Lemma \ref{a37} (for $x_1$ and $x_2$) with  Corollary \ref{d32} and Cauchy--Schwarz
\begin{align*}
\Ic_j&\lesssim N2^{-j}N^22^{-j}2^{-j}N^3\sup_{b:\;|b_n|=|a_n|}\int_{[0,1]^2\times [0,\frac1{N^3}]\times [0,\frac1{N^3}]}|\sum_{n=N/2}^Nb_ne(x_1n+x_2n^2+x_4n^4+x_5n^5)|^{12}dx\\&\lesssim_\epsilon N^{\epsilon}2^{-3j}.
\end{align*}
\\
\\
Case 5. If $2^j\ge N^4$, we use $L^2$ orthogonality in $x_5$ together with Cauchy--Schwarz ($18=2+16$)
$$\Ic_j\le 2^{-4j}N^8\sup_{b:\;|b_n|=|a_n|}\int_{[0,\frac1{N^4}]}|\sum_{n=N/2}^Nb_ne(x_5n^5)|^{2}dx_5\lesssim  2^{-4j}N^8N^{-4}\le 2^{-3j}.$$

\section{Proof of Theorem \ref{a13} }

We denote by $E$ the extension operator associated with the curve $\Phi$
\begin{equation}
\label{curve}
\Phi(t)=(t,\phi_2(t),\phi_3(t),\phi_4(t)),\;t\in[\frac12,1].
\end{equation}
More precisely, for $f:[\frac12,1]\to \C$ and $I\subset [\frac12,1]$ we write
$$E_If(x)=\int_If(t)e(tx_1+\phi_2(t)x_2+\phi_3(t)x_3+\phi_4(t)x_4)dt.$$

We recall the following results from \cite{C2}. The first one holds true since the curve $\Phi$ has torsion $\simeq 1$, as expressed by \eqref{2}.  To not obscure the presentation, we will ignore the use of weights $w_B$ throughout the rest of the paper.

\begin{thm}
	\label{49}
	Assume that $\Phi$ satisfies \eqref{1} and \eqref{2}.
	Let $I_1,I_2$ be two intervals of length $\simeq 1$ in $[\frac12,1]$, with $dist(I_1,I_2)\simeq 1$. Let also $f_i:[\frac12,1]\to \C$. Then for each ball $B_N$ of radius $N$ in $\R^4$ we have
	$$\|E_{I_1}f_1E_{I_2}f_2\|_{L^6(B_N)}\lesssim_\epsilon N^{\epsilon}(\sum_{J_1\subset I_1}\sum_{J_2\subset I_2}\|E_{J_1}f_1E_{J_2}f_2\|_{L^6(B_N)}^2)^{1/2}.$$
	The sum on the right is over intervals $J$ of length $N^{-1/2}$.
\end{thm}
We will use this in combination with the following inequality
\begin{equation}
\label{52}
\|E_{J_1}f_1E_{J_2}f_2\|_{L^6(B_N)}^6\lesssim N^{-4}\|E_{J_1}f_1\|_{L^6(B_N)}^6\|E_{J_2}f_2\|_{L^6(B_N)}^6.
\end{equation}
The third inequality we need from \cite{C2} is stated below.
\begin{thm}
	\label{65}	
	Assume $\psi_1,\ldots,\psi_4:[-1,1]\to\R$ have $C^3$ norm $O(1)$, and in addition satisfy
	$$|\psi_2''(t)|,|\psi_3''(t)|\ll 1,\; \forall\; |t|\le 1$$
	and
	$$|\psi_1''(t)|,|\psi_4''(t)|\simeq 1,\; \forall\; |t|\le 1.$$
	Let
	$$\Psi(t,s)=(t,s,\psi_1(t)+\psi_2(s),\psi_3(t)+\psi_4(s)),\;\;|t|,|s|\le 1.$$
	Then for each ball $B_N\subset \R^4$ with radius $N$ and each
	constant coefficients $c_{m_1,m_2}\in\C$ we have
	\begin{equation}
	\label{59}
	\|\sum_{m_1\le N^{1/2}}\sum_{m_2\le N^{1/2}}c_{m_1,m_2}e(x\cdot \Psi(\frac{m_1}{N^{1/2}},\frac{m_2}{N^{1/2}}))\|_{L^6(B_N)}\lesssim_{\epsilon}N^\epsilon \|c_{m_1,m_2}\|_{\ell^2}|B_N|^{1/6}.
	\end{equation}

	The implicit constant is independent of $N$ and of $\psi_i$.
\end{thm}

We now proceed with the proof of Theorem \ref{a13}. Fix $a_n$ with $\|a\|_{\ell^2}=1$.  Rescaling the last three variables,  we slightly modify the earlier notation and write
$$\Ec_{I}(x)=\sum_{n\in I}a_ne(nx_1+\phi_2(\frac{n}{N})Nx_2+\phi_3(\frac{n}{N})Nx_3+\phi_4(\frac{n}{N})Nx_4)=\sum_{n\in I}a_ne(N\Phi(\frac{n}{N})\cdot x).$$
Note that
$$\Ec_I(x)=E_{I/N}f(Nx),$$
where $f$ is the distribution equal to $$\sum_{n\in I}a_n\delta_{\frac{n}{N}}.$$
Standard approximation arguments allow Theorem \ref{49} to also be applicable to such $f$.

We also write
$$\Omega=[0,1]\times[0,N]\times[0,N]\times [0,1].$$
We need to prove that
\begin{equation}
\label{a22}
\int_\Omega|\Ec_{I_1}\Ec_{I_2}|^6\lesssim N^{2+\epsilon}.
\end{equation}

The argument involves two decouplings.
\\
\\
Step 1. We cover $\Omega$ with cubes $B$ of side length $1$,  apply Theorem \ref{49}  on each $B$ (or rather $NB$, after rescaling),
then we sum these estimates to get
$$\int_{\Omega}|\Ec_{I_1}\Ec_{I_2}|^6\lesssim_\epsilon N^{\epsilon}[\sum_{J_1\subset I_1}\sum_{J_2\subset I_2}(\int_{\Omega}|\Ec_{J_1}\Ec_{J_2}|^6)^{1/3}]^{3}.$$
Here $J_1,J_2$ are intervals of length $N^{1/2}$.

The remaining part of the argument will be concerned with proving the estimate $$\int_{\Omega}|\Ec_{J_1}\Ec_{J_2}|^6\lesssim_\epsilon N^{2+\epsilon}\|a_n\|_{\ell^2(J_1)}^6\|a_n\|_{\ell^2(J_2)}^6.$$
The combination of the last two inequalities leads to \eqref{a22}.

For $i=1,2$ fix $J_i=[h_i+1,h_i+N^{1/2}]$.
\\
\\
Step 2.
We point out the main difference between the forthcoming  argument and  the one in \cite{C2}. Here, the variables $x_2$ and $x_3$ play an entirely symmetrical role, not just in terms of range, but also functionality.

We will seek  a change of variables in $\R^4$, one that will allow us to use Theorem \ref{65}.   As in \cite{C2}, the variable $x_4$ plays no role in this part of the argument, as it produces no oscillations. This variable only played a role in the first step of the argument. We need to create another variable, in addition to $x_1,x_2,x_3$.

First, we apply \eqref{52} on each cube $NB$
\begin{align*}
\int_{B}|\Ec_{J_1}\Ec_{J_2}|^6&=N^{-4}\int_{NB}|\Ec_{J_1}(\frac{\cdot}{N})\Ec_{J_2}(\frac{\cdot}{N})|^6\\&\lesssim N^{-8}\int_{NB}|\Ec_{J_1}(\frac{\cdot}{N})|^6\int_{NB}|\Ec_{J_2}(\frac{\cdot}{N})|^6=\int_{B}|\Ec_{J_1}|^6\int_{B}|\Ec_{J_2}|^6.
\end{align*}
Second, we use the following abstract inequality, that only relies on the positivity of $|\Ec_{J_i}|^6$
\begin{equation}
\label{c3}
\sum_{B\subset \Omega}\int_{B}|\Ec_{J_1}|^6\int_{B}|\Ec_{J_2}|^6\lesssim \int_{\Omega}dx\int_{(y,z)\in [-1,1]^4\times [-1,1]^4}|\Ec_{J_1}(x+y)\Ec_{J_2}(x+z)|^6dydz.
\end{equation}
Using periodicity in the $y_1,z_1$ variables, we can dominate  the right hand side above by
$$ \frac1{N^2}\int_{x_1,x_4,y_2,y_3,y_4,z_2,z_3,z_4\in[-1,1]}dx_1\ldots dz_4\int_{y_1,z_1,x_2,x_3\in[0,N]}|\Ec_{J_1}(x+y)\Ec_{J_2}(x+z)|^6dy_1dz_1dx_2dx_3.$$

In short, the variable $x_1$ is now replaced with the new variables $y_1$ and $z_1$. It remains to prove that the following square root cancellation
\begin{equation}
\label{60}
\int_{y_1,z_1,x_2,x_3\in[0,N]}|\Ec_{J_1}(x+y)\Ec_{J_2}(x+z)|^6dy_1dz_1dx_2dx_3\lesssim_\epsilon N^{4+\epsilon}\|a_n\|_{\ell^2(J_1)}^{6}\|a_n\|_{\ell^2(J_2)}^6
\end{equation}
holds uniformly over $x_1,x_4,y_2,y_3,y_4,z_2,z_3,z_4$. With these variables fixed for the rest of the argument,
we make the linear change of variables $(y_1,z_1,x_2,x_3)\mapsto (u_1,u_2,w_1,w_2)$

\begin{equation}
\label{999999}
\begin{cases}u_1=y_1+\phi_2'(\frac{h_1}{N})x_2+\phi_3'(\frac{h_1}{N})x_3\\u_2=z_1+\phi_2'(\frac{h_2}{N})x_2+\phi_3'(\frac{h_2}{N})x_3 \\w_1=\phi_2''(\frac{h_1}{N})\frac{x_2}{2N}+\phi_3''(\frac{h_1}{N})\frac{x_3}{2N}\\w_2=\phi_2''(\frac{h_2}{N})\frac{x_2}{2N}+\phi_3''(\frac{h_2}{N})\frac{x_3}{2N}\end{cases}.
\end{equation}
The  Jacobian is $\simeq \frac1{N^2}$, due to \eqref{3}.	The cube $[0,N]^4$ is mapped to a subset of $|u_1|,|u_2|\lesssim N$, $|w_1|,|w_2|\lesssim 1$. Note also  that, due to \eqref{3}, $x_3=Aw_1+Bw_2$, $x_2=Cw_1+Dw_2$, where $A,B,C,D$ depend only on $h_1,h_2$, and $|A|,B|,|C|,|D|\lesssim N$.

Let for $i=1,2$
$$\begin{cases}\theta_i(m)={m^3}\frac{\phi_2'''(\frac{h_i}{N})}{3!N^2}+{m^4}\frac{\phi_2''''(\frac{h_i}{N})}{4!N^3}+\ldots\\ \eta_i(m)={m^3}\frac{\phi_3'''(\frac{h_i}{N})}{3!N^2}+{m^4}\frac{\phi_3''''(\frac{h_i}{N})}{4!N^3}+\ldots\end{cases}$$

Using this   we may dominate the integral in \eqref{60} by
\begin{equation}
\label{a24}
{N^2}\int_{|u_i|\lesssim N,\;{|w_i|\lesssim 1}}|\sum_{m_1=1}^{ N^{\frac12}}\sum_{m_2=1}^{ N^{\frac12}}c_{m_1,m_2}e(m_1u_1+m_1^2w_1+m_2u_2+m_2^2w_2+
\end{equation}
$$+
(\theta_1(m_1)+\theta_2(m_2))(Cw_1+Dw_2)+(\eta_1(m_1)+\eta_2(m_2))(Aw_1+Bw_2))|^6du_1du_2dw_1dw_2.
$$
The coefficient $c_{m_1,m_2}$ depends only on $m_1,m_2,x_1,y_2,z_2,y_3,z_3,x_4, y_4,z_4$, but not on the variables of integration $u_i,w_i$. Moreover,
$$|c_{m_1,m_2}|=|a_{h_1+m_1}a_{h_2+m_2}|.$$
The argument of each exponential  may be rewritten as
$$
\frac{m_1}{N^{1/2}}u_1N^{1/2}+(\psi_1(\frac{m_1}{N^{1/2}})+\psi_2(\frac{m_2}{N^{1/2}}))w_1N+$$$$\frac{m_2}{N^{1/2}}u_2N^{1/2}+(\psi_3(\frac{m_1}{N^{1/2}})+\psi_4(\frac{m_2}{N^{1/2}}))w_2N$$
where
$$\begin{cases}\psi_1(t)=t^2+&{t^3}\frac{A\phi_3'''(\frac{h_1}{N})+C\phi_2'''(\frac{h_1}{N})}{3!N^{3/2}}+{t^4}\frac{A\phi_3''''(\frac{h_1}{N})+C\phi_2''''(\frac{h_1}{N})}{4!N^{2}}+\ldots
\\\psi_2(t)=&{t^3}\frac{A\phi_3'''(\frac{h_2}{N})+C\phi_2'''(\frac{h_2}{N})}{3!N^{3/2}}+{t^4}\frac{A\phi_3''''(\frac{h_2}{N})+C\phi_2''''(\frac{h_2}{N})}{4!N^{2}}+\ldots\\\psi_3(t)=&{t^3}\frac{B\phi_3'''(\frac{h_1}{N})+D\phi_2'''(\frac{h_1}{N})}{3!N^{3/2}}+{t^4}\frac{B\phi_3''''(\frac{h_1}{N})+D\phi_2''''(\frac{h_1}{N})}{4!N^{2}}-\ldots
\\\psi_4(t)=t^2+&{t^3}\frac{B\phi_3'''(\frac{h_2}{N})+D\phi_2'''(\frac{h_2}{N})}{3!N^{3/2}}+{t^4}\frac{B\phi_3''''(\frac{h_2}{N})+D\phi_2''''(\frac{h_2}{N})}{4!N^{2}}-\ldots\end{cases}.$$
These functions satisfy the requirements in Theorem \ref{65}.  The expression in \eqref{a24} becomes
$$\frac1{N}\int_{|u_i|\lesssim N^{3/2},\;{|w_i|\lesssim N}}|\sum_{m_1=1}^{N^{1/2}}\sum_{m_2=1}^ {N^{1/2}}c_{m_1,m_2}e((u_1,u_2,w_1,w_2)\cdot\Psi(\frac{m_1}{N^{1/2}},\frac{m_2}{N^{1/2}}))|^6du_1du_2dw_1dw_2.$$
If we cover the domain of integration with balls $B_N$ and apply \eqref{59} on each of them, we may dominate the above expression by
$$N^{4+\epsilon}\|c_{m_1,m_2}\|_{\ell^2([1,N^{1/2}]\times[1,N^{1/2}])}^6=N^{4+\epsilon}\|a_n\|_{\ell^2(J_1)}^{6}\|a_n\|_{\ell^2(J_2)}^6.$$
This proves \eqref{60} and ends the argument.

\section{Proof of Theorem \ref{a31}	}

In this section we prove that Theorem \ref{a13} implies Theorem \ref{a31}.
\medskip

The parameter $K$ will be very large and universal, independent of $N$, $\phi_k$. The larger the $K$ we choose to work with, the smaller the $\epsilon$ from the $N^\epsilon$ loss will be at the end of the section.

\begin{prop}
	\label{9}
	Assume  $\phi_2,\phi_3,\phi_4: (0,3)\to\R$ are real analytic and satisfy \eqref{1}, \eqref{2}, \eqref{3} and \eqref{4} on $[1/4,1]$. Let as before $\omega_2=\omega_3=[0,N^2]$, $\omega_4=[0,N]$ and
	$$\Ec_{I,N}(x)=\sum_{n\in I}a_ne(nx_1+\phi_2(\frac{n}{N})x_2+\phi_3(\frac{n}{N})x_3+\phi_4(\frac{n}{N})x_4).$$
	We consider arbitrary integers $N_0,M$ satisfying  $1\le M\le \frac{N_0}{K}$ and $N_0+[M,2M]\subset [\frac{N}{2},N]$.
	Let $H_1,H_2$ be intervals of length $\frac{M}K$ inside $N_0+[M,2M]$ such that $\dist(H_1,H_2)\ge \frac{M}{K}$. Then
	$$\int_{[0,1]\times \omega_2\times \omega_3\times \omega_4}|\Ec_{H_1,N}(x)\Ec_{H_2,N}(x)|^6\lesssim_\epsilon N^{9+\epsilon}\|a\|^{12}_{\ell^6([N_0+M,N_0+2M])}.$$	
\end{prop}
\begin{proof}
	Write $H_1=N_0+I_1$, $H_2=N_0+I_2$ with $I_1,I_2$ intervals of length $\frac{M}K$ inside $[M,2M]$ and with separation $\ge \frac{M}{K}$. Note that $N_0/N\in [1/4,1]$. Note that the roles and the properties of $\phi_2$, $\phi_3$ are completely symmetrical in $\Ec_{I,N}$ and in \eqref{1}, \eqref{2}, \eqref{3} and \eqref{4}.
	It follows by \eqref{3} that either $\phi_2^{(2)}(\frac{N_0}{N})$	or $\phi_3^{(2)}(\frac{N_0}{N})$ is nonzero. So, due to symmetry, we can assume without the loss of generality that $\phi_2^{(2)}(\frac{N_0}{N})\neq 0$.
	
	We use the following expansion, certainly valid for all $m$ in $I_i$.
	\begin{align*}
	\phi_2(\frac{N_0+m}{N})&=Q_2(m)+\sum_{n\ge 2}\frac{\phi_2^{(n)}(\frac{N_0}{N})}{n!}\kappa^n\\&=Q_2(m)+\kappa^2\sum_{n\ge 2}\frac{\phi_2^{(n)}(\frac{N_0}{N})\kappa^{n-2}}{n!}(\frac{m}{M})^n.
	\end{align*}
	Here $Q_2(m)=A+Bm$ with $B=O(\frac1N)$, and we denoted $\kappa=M/N$. Observe that by choosing $K$ sufficiently large we  can make $\kappa$ arbitrarily small.
	We introduce the analogue $\tilde{\phi_2}$ of $\phi_2$ at scale $M$
	$$\tilde{\phi_2}(t)=\sum_{n\ge 2}\frac{\phi_2^{(n)}(\frac{N_0}{N})\kappa^{n-2}}{n!}t^n.$$
	This series is convergent as long as $\frac{N_0}{N}+t\in (0,3)$, so the new function is certainly real analytic on $(0,2)$, since $N_0\le N$.
	We can decompose $\tilde{\phi_2}$ as	
	$$\tilde{\phi_2}(t)=: a_2 t^2+a_3 t^3 \kappa+a_4 t^4 \kappa^2+r_2(t)\kappa^{3},$$
	with
	$$
	a_n=\frac{\phi_2^{(n)}(\frac{N_0}{N})}{n!},\qquad n=2,3,4,
	$$
	and $r_2(t)$ satisfying
	$$
	\sup_{k=2,3,4} \sup_{t\in[1/2,1]} |r_2^{(k)}(t)|=O(1).
	$$
	We have
	$$	\phi_2(\frac{N_0+m}{N})=Q_2(m)+\kappa^2\tilde{\phi_2}(\frac{m}{M}).$$

	We also write for  $m\in I_i$
	with $Q_3(m)=C+Dm$ satisfying $D=O(\frac1N)$,
	\begin{equation}
	\label{37e46ryfyrfuyru}
	\phi_3(\frac{N_0+m}{N})=Q_3(m)+\sum_{n\ge 2}\frac{\phi_3^{(n)}(\frac{N_0}{N})\kappa^{n}}{n!}(\frac{m}{M})^n.\end{equation}
	We will use the following formula, with $A_n=\frac{\phi_2^{(n)}(\frac{N_0}{N})\kappa^{n}}{n!}(\frac{m}{M})^n$, $B_n=\frac{\phi_3^{(n)}(\frac{N_0}{N})\kappa^{n}}{n!}(\frac{m}{M})^n$
	$$\sum_{n\ge 2}B_n=\frac{B_2}{A_2}\sum_{n\ge 2}A_n+\sum_{n\ge 3}\frac{B_nA_2-B_2A_n}{A_2},$$
	provided that $A_2\neq 0$.
	The sum $\sum_{n\ge 2}$ in \eqref{37e46ryfyrfuyru} is equal to
	$$
	\frac{\phi_3^{(2)}(\frac{N_0}{N})}{\phi_2^{(2)}(\frac{N_0}{N})}\kappa^2\tilde{\phi}_2(\frac{m}{M})+\sum_{n\ge 3}
	\frac{\phi_3^{(n)}(\frac{N_0}{N})\phi_2^{(2)}(\frac{N_0}{N})-\phi_3^{(2)}(\frac{N_0}{N})\phi_2^{(n)}(\frac{N_0}{N})}{\phi_2^{(2)}(\frac{N_0}{N})n!}\kappa^{n}(\frac{m}{M})^n.$$
	Let $\tilde{\phi}_3$ be the analogue of $\phi_3$ at scale $M$ defined by
	$$
	\tilde{\phi}_3(t)=\sum_{n\ge 3}
	\frac{\phi_3^{(n)}(\frac{N_0}{N})\phi_2^{(2)}(\frac{N_0}{N})-\phi_3^{(2)}(\frac{N_0}{N})\phi_2^{(n)}(\frac{N_0}{N})}{\phi_2^{(2)}(\frac{N_0}{N})n!}\kappa^{n-3}t^n.
	$$
	This can be decomposed as
	$$
	\tilde{\phi}_3(t)=:b_3 t^3+ b_4 t^4 \kappa+r_3(t)\kappa^{2} ,
	$$
	with
	\begin{align*}
	b_n=\frac{\phi_3^{(n)}(\frac{N_0}{N})\phi_2^{(2)}(\frac{N_0}{N})-\phi_3^{(2)}(\frac{N_0}{N})\phi_2^{(n)}(\frac{N_0}{N})}{\phi_2^{(2)}(\frac{N_0}{N})n!}, \qquad n=3,4,
	\end{align*}
	and $r_3(t)$ satisfying
	$$
	\sup_{k=2,3,4}\sup_{t\in[1/2,1]} |r_3^{(k)}(t)|=O(1).
	$$
	We can write
	$$
	\phi_3(\frac{N_0+m}{N})=Q_3(m)+\frac{\phi_3^{(2)}(\frac{N_0}{N})}{\phi_2^{(2)}(\frac{N_0}{N})}\kappa^2\tilde{\phi}_2(\frac{m}{M})+\kappa^{3}\tilde{\phi}_3(\frac{m}{M}).
	$$
	
	Finally, we let $Q_4(m)=E+Fm$ with $F=O(\frac1N)$. Note that \eqref{4} guarantees that $b_3\not=0$. This allows us to define
	\begin{align*}
	&\phi_4(\frac{N_0+m}{N})=Q_4(m)+\sum_{n\ge 2}\frac{\phi_4^{(n)}(\frac{N_0}{N})\kappa^{n}}{n!}(\frac{m}{M})^n
	\\
	&=Q_4(m)+\frac{\phi_4^{(2)}(\frac{N_0}{N})}{2! a_2}\left( a_2 (\frac{m}{M})^2+a_3 (\frac{m}{M})^3 \kappa+a_4 (\frac{m}{M})^4 \kappa^2 \right)\kappa^2
	\\
	&\quad+\frac{\frac{\phi_4^{(3)}(\frac{N_0}{N})}{3!}-\frac{a_3\phi_4^{(2)}(\frac{N_0}{N})}{2! a_2}}{b_3}\left( b_3 (\frac{m}{M})^3+b_4 (\frac{m}{M})^4 \kappa\right)\kappa^3
	\\
	&\quad+
	\bigg[\Big(-\frac{a_4 \phi_4^{(2)}(\frac{N_0}{N})}{2! a_2}+\frac{\phi_4^{(4)}(\frac{N_0}{N})}{4!}-\frac{b_4}{b_3}\left(-\frac{a_3 \phi_4^{(2)}(\frac{N_0}{N})}{2!a_2}+\frac{\phi_4^{(3)}(\frac{N_0}{N})}{3!}\right)\Big)(\frac{m}{M})^4
	\\
	&\quad+\sum_{n\ge 5}\frac{\phi_4^{(n)}(\frac{N_0}{N})\kappa^{n-4}}{n!}(\frac{m}{M})^n \bigg] \kappa^4
	\\
	&=Q_4(m)+\frac{\phi_4^{(2)}(\frac{N_0}{N})}{2! a_2}\widetilde{\phi}_2(\frac{m}{M})\kappa^2
	+\frac{\frac{\phi_4^{(3)}(\frac{N_0}{N})}{3!}-\frac{a_3\phi_4^{(2)}(\frac{N_0}{N})}{2! a_2}}{b_3}\widetilde{\phi}_3(\frac{m}{M})\kappa^3
	\\
	&\quad+
	\kappa^4\bigg[\Big(-\frac{a_4 \phi_4^{(2)}(\frac{N_0}{N})}{2! a_2}+\frac{\phi_4^{(4)}(\frac{N_0}{N})}{4!}-\frac{b_4}{b_3}\left(-\frac{a_3 \phi_4^{(2)}(\frac{N_0}{N})}{2! a_2}+\frac{\phi_4^{(3)}(\frac{N_0}{N})}{3!}\right)\Big)(\frac{m}{M})^4
	\\
	&\quad+\sum_{n\ge 5}\frac{\phi_4^{(n)}(\frac{N_0}{N})\kappa^{n-4}}{n!}(\frac{m}{M})^n -\frac{\phi_4^{(2)}(\frac{N_0}{N})}{2! a_2}r_2(\frac{m}{M})\kappa-\frac{\frac{\phi_4^{(3)}(\frac{N_0}{N})}{3!}-\frac{a_3\phi_4^{(2)}(\frac{N_0}{N})}{2! a_2}}{b_3}r_3(\frac{m}{M})\kappa\bigg]
	\\
	&=:Q_4(m)+\frac{\phi_4^{(2)}(\frac{N_0}{N})}{2! a_2}\widetilde{\phi}_2(\frac{m}{M})\kappa^2
	+\frac{\frac{\phi_4^{(3)}(\frac{N_0}{N})}{3!}-\frac{a_3\phi_4^{(2)}(\frac{N_0}{N})}{2! a_2}}{b_3}\widetilde{\phi}_3(\frac{m}{M})\kappa^3+\widetilde{\phi}_4(\frac{m}{M})\kappa^4.
	\end{align*}
	We write
	$$
	\widetilde{\phi}_4(t)=\Big(-\frac{a_4 \phi_4^{(2)}(\frac{N_0}{N})}{2! a_2}+
	\frac{\phi_4^{(4)}(\frac{N_0}{N})}{4!}-\frac{b_4}{b_3}\left(-\frac{a_3 \phi_4^{(2)}(\frac{N_0}{N})}{2! a_2}+\frac{\phi_4^{(3)}(\frac{N_0}{N})}{3!}\right)\Big)t^4+ r_4(t)\kappa,
	$$
	with
	$$
	r_4(t)=\sum_{n\ge 5}\frac{\phi_4^{(n)}(\frac{N_0}{N})\kappa^{n-5}}{n!}t^n -\frac{\phi_4^{(2)}(\frac{N_0}{N})}{2! a_2}r_2(t)-\frac{\frac{\phi_4^{(3)}(\frac{N_0}{N})}{3!}-\frac{a_3\phi_4^{(2)}(\frac{N_0}{N})}{2! a_2}}{b_3}r_3(t)
	$$
	satisfying
	$$
	\sup_{k=2,3,4} \sup_{t\in[1/2,1]} |r_4^{(k)}(t)|=O(1).
	$$
	Letting
	\begin{align*}
	R_2(t)&=a_3 t^3+a_4 t^4 \kappa+r_2(t)\kappa^{2},
	\\
	R_3(t)&=b_4 t^4 +r_3(t)\kappa
	\\
	R_4(t)&=r_4(t),
	\end{align*}
	we have
	\begin{equation}\label{R}
	\sup_{k=2,3,4} \sup_{t\in[1/2,1]} |R_2^{(k)}(t)|+\sup_{k=2,3,4} \sup_{t\in[1/2,1]} |R_3^{(k)}(t)|+\sup_{k=2,3,4} \sup_{t\in[1/2,1]} |R_4^{(k)}(t)|=O(1),
	\end{equation}
	and, after doing some basic algebra, we get
	\begin{align*}
	\widetilde{\phi_2}(t)&= \frac{1}{2!}\phi_2^{(2)}(\frac{N_0}{N}) t^2+ R_2(t)\kappa,
	\\
	\widetilde{\phi_3}(t)&=\frac{1}{3!}\frac{1}{\phi_2^{(2)}(\frac{N_0}{N})}\det \begin{bmatrix}
	\phi_2^{(2)}(\frac{N_0}{N}) & \phi_3^{(2)}(\frac{N_0}{N}) \\
	\phi_2^{(3)}(\frac{N_0}{N}) & \phi_3^{(3)}(\frac{N_0}{N}) \\
	\end{bmatrix}t^3 +R_3(t)\kappa
	\\
	\widetilde{\phi_4}(t)&=\frac{1}{4!}\det \begin{bmatrix}
	\phi_2^{(2)}(\frac{N_0}{N}) & \phi_3^{(2)}(\frac{N_0}{N}) \\
	\phi_2^{(3)}(\frac{N_0}{N}) & \phi_3^{(3)}(\frac{N_0}{N}) \\
	\end{bmatrix}^{-1} \det \begin{bmatrix}
	\phi_2^{(2)}(\frac{N_0}{N}) & \phi_3^{(2)}(\frac{N_0}{N}) & \phi_4^{(2)}(\frac{N_0}{N})\\
	\phi_2^{(3)}(\frac{N_0}{N}) & \phi_3^{(3)}(\frac{N_0}{N}) & \phi_4^{(3)}(\frac{N_0}{N})\\
	\phi_2^{(4)}(\frac{N_0}{N}) & \phi_3^{(4)}(\frac{N_0}{N}) & \phi_4^{(4)}(\frac{N_0}{N})\\
	\end{bmatrix} t^4 +R_4(t)\kappa.
	\end{align*}
	Summarizing, we have obtained the following decomposition
	\begin{align*}
	\phi_2(\frac{N_0+m}{N})&=Q_2(m)+\kappa^2\tilde{\phi_2}(\frac{m}{M}),
	\\
	\phi_3(\frac{N_0+m}{N})&=Q_3(m)+\frac{\phi_3^{(2)}(\frac{N_0}{N})}{\phi_2^{(2)}(\frac{N_0}{N})}\kappa^2\tilde{\phi}_2(\frac{m}{M})+\kappa^{3}\tilde{\phi}_3(\frac{m}{M}),
	\\
	\phi_4(\frac{N_0+m}{N})&=Q_4(m)+\frac{\phi_4^{(2)}(\frac{N_0}{N})}{\phi_2^{(2)}(\frac{N_0}{N})}\kappa^2\tilde{\phi}_2(\frac{m}{M})
	+\frac{\phi_4^{(3)}(\frac{N_0}{N})\phi_2^{(2)}(\frac{N_0}{N})-\phi_2^{(3)}(\frac{N_0}{N})\phi_4^{(2)}(\frac{N_0}{N})}{\phi_3^{(3)}(\frac{N_0}{N})\phi_2^{(2)}(\frac{N_0}{N})-\phi_3^{(2)}(\frac{N_0}{N})\phi_2^{(3)}(\frac{N_0}{N})}\kappa^3\widetilde{\phi}_3(\frac{m}{M})
	\\
	&\quad+\kappa^4\widetilde{\phi}_4(\frac{m}{M}).
	\end{align*}
	It motivates the change of variables
	$$
	\begin{cases}
	y_1=x_1+Bx_2+Dx_3+Fx_4,
	\\
	y_2=\kappa^2\left (x_2+\frac{\phi_3^{(2)}(\frac{N_0}{N})}{\phi_2^{(2)}(\frac{N_0}{N})}x_3+\frac{\phi_4^{(2)}(\frac{N_0}{N})}{\phi_2^{(2)}(\frac{N_0}{N})}x_4\right),
	\\
	y_3=\kappa^3\left(x_3+\frac{\phi_4^{(3)}(\frac{N_0}{N})\phi_2^{(2)}(\frac{N_0}{N})-\phi_2^{(3)}(\frac{N_0}{N})\phi_4^{(2)}(\frac{N_0}{N})}{\phi_3^{(3)}(\frac{N_0}{N})\phi_2^{(2)}(\frac{N_0}{N})-\phi_3^{(2)}(\frac{N_0}{N})\phi_2^{(3)}(\frac{N_0}{N})}x_4\right),
	\\
	y_4=\kappa^4x_4.
	\end{cases}
	$$
	Due to  periodicity, we may extend the range of $x_1$ to $[0,N]$. This linear transformation maps $[0,N]\times \omega_1\times \omega_3\times \omega_4$ to a subset of the box $\tilde{\omega}_1\times\tilde{\omega}_2\times\tilde{\omega}_3\times\tilde{\omega}_4$ centered at the origin, with dimensions roughly $N,M^2, M^3N^{-1}, M^4N^{-3}$.
	
	Thus
	$$|\Ec_{H_k,N}(x)|=|\Ec_{I_k,M}(y)|$$
	where
	$$\Ec_{I_k,M}(y)=\sum_{m\in I_k}a_{N_0+m}e(my_1+\tilde{\phi}_2(\frac{m}{M})y_2+\tilde{\phi}_3(\frac{m}{M})y_3+\tilde{\phi}_4(\frac{m}{M})y_4).$$

	Note that, as we mentioned before, by choosing $K$ large enough, $\kappa$ can be made arbitrarily small. Therefore, the functions $\tilde{\phi_2}, \tilde{\phi_3}$ and $\tilde{\phi_4}$ satisfy conditions \eqref{1}, \eqref{2} and \eqref{3}. Indeed, condition \eqref{1} follows from \eqref{R} and the fact that functions $\phi_2, \phi_3$ and $\phi_4$ satisfy \eqref{1}. Condition \eqref{2} follows from the identity
	\begin{align*}
	\det \begin{bmatrix}
	\tilde{\phi_2}^{(2)}(\frac{N_0}{N}) & \tilde{\phi_3}^{(2)}(\frac{N_0}{N}) & \tilde{\phi_4}^{(2)}(\frac{N_0}{N})\\
	\tilde{\phi_2}^{(3)}(\frac{N_0}{N}) & \tilde{\phi_3}^{(3)}(\frac{N_0}{N}) & \tilde{\phi_4}^{(3)}(\frac{N_0}{N})\\
	\tilde{\phi_2}^{(4)}(\frac{N_0}{N}) & \tilde{\phi_3}^{(4)}(\frac{N_0}{N}) & \tilde{\phi_4}^{(4)}(\frac{N_0}{N})\\
	\end{bmatrix}
	=
	\frac{1}{2! 3! 4!}\det \begin{bmatrix}
	\phi_2^{(2)}(\frac{N_0}{N}) & \phi_3^{(2)}(\frac{N_0}{N}) & \phi_4^{(2)}(\frac{N_0}{N})\\
	\phi_2^{(3)}(\frac{N_0}{N}) & \phi_3^{(3)}(\frac{N_0}{N}) & \phi_4^{(3)}(\frac{N_0}{N})\\
	\phi_2^{(4)}(\frac{N_0}{N}) & \phi_3^{(4)}(\frac{N_0}{N}) & \phi_4^{(4)}(\frac{N_0}{N})\\
	\end{bmatrix}+O(\kappa).
	\end{align*}
	Finally, a direct computation reveals that
	\begin{align*}
	\det \begin{bmatrix}
	\tilde{\phi_2}''(t)&\tilde{\phi_3}''(t)\\
	\tilde{\phi_2}''(s)&\tilde{\phi_3}''(s)
	\end{bmatrix}=
	|t-s|\det\begin{bmatrix}
	{\phi_2}''(\frac{N_0}{N})&{\phi_3}''(\frac{N_0}{N})\\
	{\phi_2}''(\frac{N_0}{N})&{\phi_3}''(\frac{N_0}{N})
	\end{bmatrix},
	\end{align*}
	which implies \eqref{3}.
	
	We may write, using again periodicity in $y_1$
	\begin{align*}
	\int_{[0,1]\times \omega_2\times \omega_3\times \omega_4}|\Ec_{H_1,N}(x)\Ec_{H_2,N}(x)|^6&=\frac1{N}\int_{[0,N]\times \omega_2\times \omega_3\times \omega_4}|\Ec_{H_1,N}(x)\Ec_{H_2,N}(x)|^6\\&\le(\frac{N}{M})^9\int_{[0,1]\times\tilde{\omega}_2\times\tilde{\omega}_3\times\tilde{\omega_4}}|\Ec_{I_1,M}(y)\Ec_{I_2,M}(y)|^6.
	\end{align*}
	Finally, we use Theorem \ref{4} with $N=M$, noting  that  $\tilde{\omega}_2,\tilde{\omega}_3\subset [-M^2,M^2]$	and $\tilde{\omega}_4\subset [-M,M]$, to estimate the last expression by
	$$(\frac{N}{M})^9M^{5+\epsilon}\|a\|^{12}_{\ell^2([N_0+M,N_0+2M])}\le N^{9+\epsilon}\|a\|^{12}_{\ell^6([N_0+M,N_0+2M])}.$$
\end{proof}	
\smallskip

We can now prove  Theorem \ref{a31}. Let $\Omega=[0,1]\times [0,N^2]\times [0,N^2]\times [0,N]$.

Let $\Hc_n(I)$ be the collection of dyadic intervals in $I$ with length $\frac{N}{2K^n}$. We write $H_1\not\simeq H_2$ to imply that $H_1,H_2$ are not neighbors. Then
$$|\Ec_{I,N}(x)|\le 3\max_{H\in \Hc_1(I)}|\Ec_{H,N}(x)|+K^{10}\max_{H_1\not\simeq H_2\in\Hc_1(I)}|\Ec_{H_1,N}(x)\Ec_{H_2,N}(x)|^{1/2}.$$
We repeat this inequality until we reach intervals in $\Hc_{l}$ of length $\simeq 1$,  that is  $K^l\simeq N$.	We have
\begin{align*}
|\Ec_{I,N}(x)|&\lesssim l3^lK^{10}\max_{1\le n\le l}\max_{H\in\Hc_n(I)}\max_{H_1\not\simeq H_2\in\Hc_{n+1}(H)}|\Ec_{H_1,N}(x)\Ec_{H_2,N}(x)|^{1/2}\\&\lesssim (\log N)N^{\log_K3}\max_{1\le n\le l}\max_{H\in\Hc_n(I)}\max_{H_1\not\simeq H_2\in\Hc_{n+1}(H)}|\Ec_{H_1,N}(x)\Ec_{H_2,N}(x)|^{1/2}.
\end{align*}
Using Proposition \ref{9}  we finish the proof as follows
\begin{align*}
\int_{\Omega} |\Ec_{I,N}(x)|^{12}dx&\lesssim_KN^{\epsilon+O(\log_K 3)}\sum_{n}\sum_{H\in\Hc_n(I)}\max_{H_1\not\simeq H_2\in\Hc_{n+1}(H)}\int_{\Omega}|\Ec_{H_1,N}(x)\Ec_{H_2,N}(x)|^{6}dx\\&\lesssim_{K,\epsilon}N^{\epsilon+O(\log_K 3)}\sum_n\sum_{H\in\Hc_n(I)}N^{9}\|a_n\|_{\ell^6(H)}^{12}
\\&\lesssim_{K,\epsilon}N^{9+\epsilon+\log_K 3}\|a\|_{\ell^6}^{12}.
\end{align*}
Choosing $K$ large enough, we may force $\log_K 3$ to be as small as we wish.

\section{Proof of Theorem \ref{c7}}
\label{d=5last}

We will work with the functions
$$\phi_1(t)=t,\;\phi_2(t)=t^2,\;\phi_3(t)=t^3+\epsilon_3t^4,\;\phi_4(t)=t^4+\epsilon_4t^5,\;\phi_5(t)=t^5,$$
where $\epsilon_3,\epsilon_4=o(1)$. The smallness of $\epsilon_3,\epsilon_4$ will be used in the proof of Proposition \ref{c5}.

For all practical purposes, $\Phi=(\phi_1,\ldots,\phi_5)$ will be a negligible perturbation (in fact, a nonsingular linear image) of the moment curve $\Phi_0$ corresponding to $\epsilon_3=\epsilon_4=0$.
All implicit constants in the results that we prove about $\Phi$ will be uniform over all such $\epsilon_i$.

We abuse earlier notation and write
$$\Ec_{I,N}(x)=\sum_{n\in I}a_ne(\Phi(\frac{n}{N})\cdot x).$$
Note that this is $N$-periodic in $x_1$ and $N^2$-periodic in $x_2$.
At the end of this section we prove that Theorem \ref{c7} is a consequence of the following bilinear result.
\begin{thm}[Bilinear small cap $l^{9}L^{18}$  decoupling]
	\label{b13}Let $I_1,I_2$ be intervals of length $\simeq N$ in $[\frac{N}{2},N]$, with $\dist(I_1,I_2)\simeq {N}$.	
	Let $\Omega=[0,N]\times[0,N^2]\times\omega_3\times\omega_4\times\omega_5$, where $\omega_i$ are intervals satisfying $|\omega_3|, |\omega_4|\ge N^3$, $|\omega_5|\ge N$. Then we have
	$$\int_{\Omega}|\prod_{j=1}^2\Ec_{I_j,N}(x)|^{9}dx\lesssim_\epsilon N^{18(\frac12-\frac1{9})+\epsilon}|\Omega|\|a_n\|_{\ell^{9}(I_1)}^{9}\|a_n\|_{\ell^{9}(I_2)}^{9}.$$ 	
\end{thm}
From now on, we may and will assume that
$$\Omega=[0,N^3]^4\times [0,N].$$
Enlarging the range of the	 first two variables is done only for convenience, to accommodate various changes of variables. It comes at no cost, due to periodicity. The novel small cap decoupling nature of this result is reflected by the size $\simeq N$ of the range of $x_5$. Our argument cannot accommodate a smaller range, due to our crucial use of $N$-periodicity in $x_1$ for $\Ec_{I,N}$. This will be apparent in Step 2 of the argument. However, the size $\simeq N$ is exactly what is needed in our applications.

At the heart of our proof of Theorem \ref{b13} lies the following inequality proved in \cite{G}, for the surface
$$\Gc=(t,s,t^2,s^2,t^3+s^3).$$
 This can be thought of as a two dimensional analog of the decoupling for the moment curve in $\R^3$.  It will play the same role in this section as the role played by $L^6$ decoupling in the proof of Theorem \ref{a13}. Notably, the argument in this section will make crucial use of the oscillatory nature of the cubic terms in $\Gc$.

\begin{thm}
\label{guo}	
Given $f:[0,1]^2\to\C$ and intervals $J_1,J_2\subset [0,1]$
let $$E^{\Gc}_{J_1\times J_2}f(x)=\int_{J_1\times J_2}f(t,s)e(tx_1,sx_2,t^2x_3,s^2x_4,(t^3+s^3)x_5)dtds.$$
Then for  each ball $B_R\subset \R^5$ we have
$$\|E^{\Gc}_{[0,1]^2}f\|_{L^{9}(B_R)}\lesssim_\epsilon R^{\frac23(\frac12-\frac19)+\epsilon}(\sum_{|J_1|,|J_2|=1/R^{1/3}}\|E^{\Gc}_{J_1\times J_2}f\|^9_{L^{9}(B_R)})^{1/9}.$$
\end{thm}

The first application of this theorem will be used to produce the Step 1 decoupling in the proof of Theorem \ref{b13}. For $f:I\subset [0,1]\to\C$ let
$$E^{\Phi}_If(x)=\int_If(t)e(\Phi(t)\cdot x)dt.$$
\begin{prop}
\label{b1}	We have
$$\|E^{\Phi}_{I_1}fE^{\Phi}_{I_2}f\|_{L^9(B_R)}\lesssim_\epsilon R^{\frac23(\frac12-\frac19)+\epsilon}(\sum_{J_i\subset I_i:\;|J_i|=1/R^{1/3}}\|E^{\Phi}_{J_1}fE^{\Phi}_{J_2}f\|^9_{L^{9}(B_R)})^{1/9}.$$
\end{prop}
The idea behind this result is that $\Phi(I_1)+\Phi(I_2)$ is a surface that can be locally approximated by nonsingular affine images of the reference surface $\Gc$. This is the approach taken in \cite{Oh}, and we refer the reader to this paper for details.

We will combine this with the following transversality result.
\begin{prop}
\label{c4}	We have
$$\|E^{\Phi}_{J_1}fE^{\Phi}_{J_2}f\|^9_{L^{9}(B_R)}\lesssim R^{-5} \|E^{\Phi}_{J_1}f\|^9_{L^{9}(B_R)}\|E^{\Phi}_{J_2}f\|^9_{L^{9}(B_R)}.$$
\end{prop}
\begin{proof}
The proof is very similar to the one for \eqref{52}, explained in \cite{C2}. We sketch the details, with an emphasis on the main geometric inequality. Since
$$E_J^{\Phi}f(x)=E_J^{\Phi_0}f(Ax),\;\; A(x_1,\ldots,x_5)=(x_1,x_2,x_3,x_4+\epsilon_3x_3,x_5+\epsilon_4x_4)$$
and since $B_R$ does not change much under the action of $A$, we may assume that $\Phi=\Phi_0$.

Let $\eta_R$ be a positive, smooth approximation of $1_{B_R}$ with Fourier support inside $B(0,1/R)$.  The Fourier transform of $\eta_RE^{\Phi_0}_{J_i}f$ is supported on the set (for arbitrary $t_i\in J_i$)
$$\{(t_i+s,\ldots,(t_i+s)^5):\;|s|=O(R^{-1/3})\}+B(0,1/R).$$
This is easily seen to lie inside a translate of the rectangular box $B(J_i)$ defined as follows. Let $\pi(t_i)$ be the plane spanned by the vectors  $$e_1(t_i)=(1,2t_i,3t_i^2,4t_i^3,5t_i^4),\;\;e_2(t_i)=(0,1,3t_i,6t_i^2,10t_i^3).$$
Let $R(t_i)$ be the rectangle inside $\pi(t_i)$, centered at the origin, with long side of length $O(1/R^{1/3})$ in the direction $e_1(t_i)$, and short side of length $O(1/R^{2/3})$ in the orthogonal direction. We take $B(J_i)$ to be the Cartesian product of $R(J_i)$ and the cube $[-O(1/R),O(1/R)]^3$, the latter being a subset of  $\pi(t_j)^\perp$.
 We write
$$|E^{\Phi_0}_{J_i}f|^91_{B_R}\approx |E^{\Phi_0}_{J_i}f|^9\eta_R\approx\sum_{P_i\in \Pc_i}c_{P_i}1_{P_i},$$
with $c_{P_i}\in (0,\infty)$, and $\Pc_i$ a tiling of $\R^5$ with rectangular boxes $P_i$ dual to $B(J_i)$. Each $P_i$ has dimensions $R^{1/3}\times R^{2/3}\times R\times R\times R$, with the first two entries corresponding to $\pi(t_i)$ and the last three corresponding to $\pi(t_i)^\perp$.

Note that if $x\in P_i$ then
$$\langle x,e_1(t_i)\rangle=O(R^{1/3}),\;\;\langle x,e_2(t_i)\rangle=O(R^{2/3}).$$

Let us describe the intersection of $P_1$ and $P_2$.
Since
$$\det\begin{bmatrix}1&2t_1&3t_1^2&4t_1^3\\1&2t_2&3t_2^2&4t_2^3\\0&1&3t_1&6t_1^2\\0&1&3t_2&6t_2^2\end{bmatrix}\simeq |t_1-t_2|^4\simeq 1,$$
it follows that $\pi(t_1)\cap \pi(t_2)=\{0\}$ and that $\pi(t_1)^\perp\cap \pi(t_2)^\perp$ is a line, spanned by some unit vector $\textbf{v}$. It also follows that the matrix $M$ with rows one through five consisting of the vectors $e_1(t_1),e_1(t_2),e_2(t_1),e_2(t_2),\textbf{v}$ has determinant of magnitude $\simeq 1$.
Since each $x\in P_1\cap P_2$ must satisfy $Mx\in [-O(R^{1/3}),O(R^{1/3})]^2\times[-O(R^{2/3}),O(R^{2/3})]^2\times [-O(R),O(R)]$, we conclude that $|P_1\cap P_2|\lesssim R^3$. This estimate can also be seen to be sharp. Thus
\begin{align*}
\|E^{\Phi}_{J_1}fE^{\Phi}_{J_2}f\|^9_{L^{9}(B_R)}&\approx \sum_{P_1\in\Pc_1}\sum_{P_2\in\Pc_2}c_{P_1}c_{P_2}|P_1\cap P_2|\\&\lesssim R^{-5}\sum_{P_1\in\Pc_1}\sum_{P_2\in\Pc_2}c_{P_1}c_{P_2}|P_1||P_2|\\&\approx  R^{-5}\|E^{\Phi}_{J_1}f\|^9_{L^{9}(B_R)}\|E^{\Phi}_{J_2}f\|^9_{L^{9}(B_R)}.
\end{align*}

\end{proof}

The following result will allow us to perform the Step 2 decoupling in the proof of Theorem \ref{b13}. We let $M=N^{2/3}$ and $H_i=[h_i+1,h_i+M]$, with $h_1,h_2,|h_1-h_2|\simeq N$. Note that the variable $x_5$ no longer plays any role. Let us write
$$\Ec_{H,short}(x_1,x_2,x_3,x_4)=\sum_{n\in H}a_ne(\phi_1(\frac{n}{N})x_1+\ldots+\phi_4(\frac{n}{N})x_4).$$
\begin{prop}
\label{c5}	We have
$$\int_{[0,N^3]^5}|\Ec_{H_1,short}(y_1,x_2,x_3,x_4)\Ec_{H_2,short}(z_1,x_2,x_3,x_4)|^9dy_1dz_1dx_2dx_3dx_4$$$$\lesssim_\epsilon N^{15+\epsilon}M^{18(\frac12-\frac1{9})}\|a_n\|^{9}_{\ell^{9}(H_1)}\|a_n\|^{9}_{\ell^{9}(H_2)}.$$	
\end{prop}
\begin{proof}
Let $\alpha_i=h_i/N$.
We make the change of variables
$$\begin{cases}
u_1=\frac1N(y_1+2\alpha_1x_2+\phi_3'(\alpha_1)x_3+\phi_4'(\alpha_1)x_4)\\
u_2=\frac1N(z_1+2\alpha_2x_2+\phi_3'(\alpha_2)x_3+\phi_4'(\alpha_2)x_4)\\w_1=\frac1{N^2}(x_2+\frac{\phi_3''(\alpha_1)}{2}x_3+\frac{\phi_4''(\alpha_1)}{2}x_4)\\w_2=\frac1{N^2}(x_2+\frac{\phi_3''(\alpha_2)}{2}x_3+\frac{\phi_4''(\alpha_2)}{2}x_4)\\v=\frac{x_2}{N^2}
\end{cases}.$$
It has Jacobian $\simeq N^{-8}$, since $\epsilon_3,\epsilon_4$ were chosen to be small.
Note that
$$|\Ec_{H_1,short}(y_1,x_2,x_3,x_4)|=|\sum_{m_1=1}^{M}a_{h_1+m_1}\times$$$$e(m_1u_1+m_1^2w_1+m_1^3\frac{\phi_3^{(3)}(\alpha_1)x_3+\phi_4^{(3)}(\alpha_1)x_4}{3!N^3}+m_1^4\frac{\phi_3^{(4)}(\alpha_1)x_3+\phi_4^{(4)}(\alpha_1)x_4}{4!N^4}+m_1^5\frac{\phi_4^{(5)}(\alpha_1)x_4}{5!N^5})|$$

$$|\Ec_{H_2,short}(z_1,x_2,x_3,x_4)|=|\sum_{m_2=1}^{M}a_{h_2+m_2}\times$$$$e(m_2u_2+m_2^2w_2+m_2^3\frac{\phi_3^{(3)}(\alpha_2)x_3+\phi_4^{(3)}(\alpha_2)x_4}{3!N^3}+m_2^4\frac{\phi_3^{(4)}(\alpha_2)x_3+\phi_4^{(4)}(\alpha_2)x_4}{4!N^4}+m_2^5\frac{\phi_4^{(5)}(\alpha_2)x_4}{5!N^5})|.$$
Using the equations for $w_1,w_2,v$
and the fact that
$$\beta:=\det\begin{bmatrix}\phi_3''(\alpha_1)&\phi_4''(\alpha_1)\\\phi_3''(\alpha_2)&\phi_4''(\alpha_2)\end{bmatrix}\simeq 1,$$
we find that
\begin{equation}
\label{b5}
x_3=aw_1+bw_2+cv,\;\;x_4=dw_1+ew_2+fv,
\end{equation}
with $a,\ldots,f=O(N^2)$.
Moreover,
$$c=2N^2\beta^{-1}[\phi_4''(\alpha_1)-\phi_4''(\alpha_2)],\;\;f=2N^2\beta^{-1}[\phi_3''(\alpha_2)-\phi_3''(\alpha_1)].$$
The coefficients of $m_1^3$ and $m_2^3$ become
$$\frac{\phi_3^{(3)}(\alpha_1)x_3+\phi_4^{(3)}(\alpha_1)x_4}{6N^3}=\frac{\epsilon_1-24(\alpha_1-\alpha_2)^2}{N\beta}v+O(1/N)w_1+O(1/N)w_2$$
and
$$\frac{\phi_3^{(3)}(\alpha_2)x_3+\phi_4^{(3)}(\alpha_2)x_4}{6N^3}=\frac{\epsilon_2+24(\alpha_1-\alpha_2)^2}{N\beta}v+O(1/N)w_1+O(1/N)w_2.$$
The numbers $\epsilon_1$, $\epsilon_2$ depend on $\epsilon_3$, $\epsilon_4$ and can be guaranteed to be as small as needed, by choosing $\epsilon_3,\epsilon_4=o(1)$. Since $|\alpha_1-\alpha_2|\simeq 1$, we have that
$$|A:=\frac{\epsilon_1-24(\alpha_1-\alpha_2)^2}{\beta}|\simeq 1,\;\;|B:=\frac{\epsilon_2+24(\alpha_1-\alpha_2)^2}{\beta}|\simeq 1.$$
We next replace $x_3,x_4$ using \eqref{b5}. We also rescale
$$\bar{u}_i=Mu_i,\;\bar{w}_i=M^2w_i,\;\bar{v}=M^3N^{-1}v.$$
This allows us to rewrite
$$|\Ec_{H_1,short}(y_1,x_2,x_3,x_4)\Ec_{H_2,short}(z_1,x_2,x_3,x_4)|=|\sum_{m_1=1}^{M}\sum_{m_2=1}^{M}a_{h_1+m_1}a_{h_2+m_2}\times$$
$$e(\frac{m_1}{M}\bar{u}_1+\frac{m_2}{M}\bar{u}_2+(\psi_1(\frac{m_1}{M})+\psi_2(\frac{m_2}{M}))\bar{w}_1+(\psi_3(\frac{m_1}{M})+\psi_4(\frac{m_2}{M}))\bar{w}_2+(\psi_5(\frac{m_1}{M})+\psi_6(\frac{m_2}{M}))\bar{v})|,$$
where $$\psi_1(t)=t^2+O(M^{-1/2})t^3+O(M^{-1})t^4+O(M^{-3/2})t^5,$$
$$\psi_2(s)=O(M^{-1/2})s^3+O(M^{-1})s^4+O(M^{-3/2})s^5,$$
$$\psi_3(t)=O(M^{-1/2})t^3+O(M^{-1})t^4+O(M^{-3/2})t^5,$$
$$\psi_4(s)=s^2+O(M^{-1/2})s^3+O(M^{-1})s^4+O(M^{-3/2})s^5,$$
$$\psi_5(t)=At^3+O(M^{-1/2})t^4+O(M^{-1})t^5,$$
$$\psi_6(s)=Bs^3+O(M^{-1/2})s^4+O(M^{-1})s^5.$$
The coefficients can be easily found, but only their size matters. It is important that $|A|,|B|\simeq 1$, and also that the leading coefficients of $\psi_1$ and $\psi_4$ have magnitude $\simeq 1$.  The result now follows immediately from Theorem \ref{b65}, that we prove below.

\end{proof}

This is the analog of Theorem \ref{65} in five dimensions.

\begin{thm}
	\label{b65}	
	Let $\beta>0$ be fixed.
	Assume that $\psi_1,\ldots,\psi_6:[-1,1]\to\R$ have $C^4$ norm $O(1)$, and in addition satisfy, uniformly over all $|t|\le 1$
	$$|\psi_2''(t)|,|\psi_3''(t)|\ll 1,$$
	$$|\psi_1''(t)|,|\psi_4''(t)|\simeq1,$$
	$$|\psi_i'''(t)|\lesssim M^{-\beta},\;1\le i\le 4, $$
	$$|\psi_5'''(t)|,|\psi_6'''(t)|\simeq1,$$
	$$|\psi_i''''(t)|\lesssim M^{-\beta},\;1\le i\le 6.$$
	Let
	$$\Psi(t,s)=(t,s,\psi_1(t)+\psi_2(s),\psi_3(t)+\psi_4(s),\psi_5(t)+\psi_6(s)),\;\;|t|,|s|\le 1.$$
	Then for each ball $B_{M^3}\subset \R^5$ with radius $M^3$ and each
	constant coefficients $c_{m_1,m_2}\in\C$ we have
	\begin{equation}
	\label{b59}
	\|\sum_{m_1\le M}\sum_{m_2\le M}c_{m_1,m_2}e(x\cdot \Psi(\frac{m_1}{M},\frac{m_2}{M}))\|_{L^9(B_{M^3})}\lesssim_{\epsilon}M^{2(\frac12-\frac19)+\epsilon} \|c_{m_1,m_2}\|_{\ell^9}|B_{M^3}|^{1/9}.
	\end{equation}
	\end{thm}
\begin{proof}
The proof is very similar to the one of Theorem \ref{65}.
The upper bound $M^{-\beta}$  may easily be relaxed to $O(1)$. We chose to work with the former in order to simplify the exposition.
Let $0<\alpha<1$ be such that $\beta+3\alpha\ge 3$.
Let
$$E^{\Psi}_Sf(x)=\int_Sf(t,s)e(x\cdot\Psi(t,s))dtds,$$
where $S\subset[-1,1]^2$. It suffices to prove that the smallest constant $d(M)$ that makes the following inequality true for each ball  $B_{M^3}$ and each $f$, satisfies $d(M)\lesssim_\epsilon M^\epsilon$
$$\|E^{\Psi}_{[-1,1]^2}f\|_{L^9(B_{M^3})}\le d(M)M^{2(\frac12-\frac19)}(\sum_{H_1,H_2\subset[-1,1]:\;|H_i|=1/M}\|E^{\Psi}_{H_1\times H_2}f\|^9_{L^9(B_{M^3})})^{1/9}.$$
This will follow once we prove that $d(M)\lesssim_\epsilon M^\epsilon d(M^{\alpha})$. First, we observe that
\begin{equation}
\label{idochurpf[]-ps]=wpd]1}
\|E^{\Psi}_{[-1,1]^2}f\|_{L^9(B_{M^3})}\le d(M^{\alpha})M^{2\alpha(\frac12-\frac19)}(\sum_{K_1,K_2\subset[-1,1]:\;|K_i|=1/M^{\alpha}}\|E^{\Psi}_{K_1\times K_2}f\|^9_{L^9(B_{M^3})})^{1/9}.
\end{equation}
Second, we note that using Taylor's expansions with fourth order remainder, the restriction of $\Psi$ to $K_1\times K_2:=[t_0-1/M^{\alpha},t_0+1/M^{\alpha} ]\times [s_0-1/M^{\alpha},s_0+1/M^{\alpha} ]$ can be written as
$$\Psi(t_0+t,s_0+s)=\Psi(t_0,s_0)+\Psi^*(t,s)+O(1/M^3),$$
where
$$
\begin{cases}
\Psi_1^*(t,s)=t,\;\;\Psi_2^*(t,s)=s\\
\Psi_3^*(t,s)=A_3t+B_3t^2+D_3s+E_3s^2\\\Psi_4^*(t,s)=A_4t+B_4t^2+D_4s+E_4s^2\\\Psi_5^*(t,s)=A_5t+B_5t^2+D_5s+E_5s^2+C_5t^3+F_5s^3
\end{cases}.
$$
Thus, $\Psi(K_1\times K_2)$ and (a translate of) $\Psi^*((K_1-t_0)\times (K_2-s_0))$ are within $O(1/M^3)$, and we may replace $\Psi$ with $\Psi^*$ on $B_{M^3}$.
Our hypothesis implies that all coefficients are $O(1)$, and moreover, $E_3, B_4=o(1)$, and also $|B_3|,|E_4|\simeq1.$ These guarantee that the Jacobian of the following linear transformation is $\simeq1$
$$
\begin{bmatrix}y_1\\y_2\\y_3\\y_4\\y_5\end{bmatrix}=\begin{bmatrix}
1&0&A_3&A_4&A_5\\
0&1&D_3&D_4&D_5\\0&0&B_3&B_4&B_5\\0&0&E_3&E_4&E_5\\0&0&0&0&1\end{bmatrix}\begin{bmatrix}x_1\\x_2\\x_3\\x_4\\x_5\end{bmatrix}.
$$
This allows us to write
$$|E^{\Psi^*}_{[-M^{-\alpha},M^{-\alpha}]^2}f(x)|\simeq|\int_{[-M^{-\alpha},M^{-\alpha}]^2}f(t,s)(y\cdot(t,s,t^2,s^2,C_5t^3+F_5s^3))dtds|.$$
Recall that our hypothesis also implies that $|C_5|,|F_5|\simeq1$. This allows us to apply Theorem \ref{guo}, which upon rescaling gives
\begin{equation}
\label{idochurpf[]-ps]=wpd]}
\|E^{\Psi}_{K_1\times K_2}f\|_{L^9(B_{M^3})}\lesssim_\epsilon M^{2(1-\alpha)(\frac12-\frac19)+\epsilon}(\sum_{H_i\subset K_i:\;|H_i|=1/M}\|E^{\Psi}_{H_1\times H_2}f\|^9_{L^9(B_{M^3})})^{1/9}.\end{equation}
Finally, the combination of \eqref{idochurpf[]-ps]=wpd]1} and \eqref{idochurpf[]-ps]=wpd]} shows that $d(M)\lesssim_\epsilon M^\epsilon d(M^{\alpha})$.

\end{proof}
It is time to prove Theorem \ref{b13}. Recall that we work with  $\Omega=[0,N^3]^4\times [0,N].$
\begin{proof}
Step 1. We cover $\Omega$ with cubes $B$ of side length $N$, and apply Proposition \ref{b1} on each of them, then sum up the inequalities to get
	
	\begin{equation}
	\label{c6}
	\int_{\Omega}|\prod_{j=1}^2\Ec_{I_j,N}(x)|^{9}dx\lesssim_\epsilon N^{6(\frac12-\frac19)+\epsilon}\sum_{J_1\subset I_1}\sum_{J_2\subset I_2}\int_{\Omega}|\prod_{j=1}^2\Ec_{J_j,N}(x)|^{9}dx.	\end{equation}
The intervals $J_i$ have length $M=N^{2/3}$. 		
\\
\\
Step 2. Fix $J_1,J_2$. We apply Proposition \ref{c4} on each $B$,  followed by the analog of the smoothing inequality \eqref{c3}
\begin{align*}
\int_{\Omega}|\prod_{j=1}^2\Ec_{J_j,N}(x)|^{9}dx&\lesssim N^{-5}\sum_{B\subset \Omega}\int_{B}|\Ec_{J_1,N}(x)|^6dx\int_{B}|\Ec_{J_2,N}(x)|^6dx\\&\lesssim N^{-10}\int_{\Omega}dx\int_{(y,z)\in [-N,N]^5\times [-N,N]^5}|\Ec_{J_1,N}(x+y)\Ec_{J_2,N}(x+z)|^6dydz.
\end{align*}
We freeze the variables $x_1, x_5,y_2,y_3,y_4,y_5,z_2,z_3,z_4,z_5$, whose range has volume $\simeq N^{12}$. We hide  their contributions into the coefficient $a_n$, whose magnitude remains the same. Recall our earlier notation $$\Ec_{H,short}(x_1,x_2,x_3,x_4)=\sum_{n\in H}a_ne(\phi_1(\frac{n}{N})x_1+\ldots+\phi_4(\frac{n}{N})x_4).$$
Using $N$-periodicity in the variables $y_1,z_1$,  the previous expression can be dominated by
$$N^{-2}\int_{[0,N^3]^5}|\Ec_{H_1,short}(y_1,x_2,x_3,x_4)\Ec_{H_2,short}(z_1,x_2,x_3,x_4)|^9dy_1dz_1dx_2dx_3dx_4. $$
By Proposition \ref{c5}, this is at most
$$ N^{13+\epsilon}N^{12(\frac12-\frac1{9})}\|a_n\|^{9}_{\ell^{9}(J_1)}\|a_n\|^{9}_{\ell^{9}(J_2)}= N^{17+\frac23}\|a_n\|^{9}_{\ell^{9}(J_1)}\|a_n\|^{9}_{\ell^{9}(J_2)}.$$
Finally, we combine this with \eqref{c6} to get the desired estimate
\begin{align*}
\int_{\Omega}|\prod_{j=1}^2\Ec_{I_j,N}(x)|^{9}dx&\lesssim_\epsilon N^{6(\frac12-\frac19)+\epsilon}\sum_{J_1\subset I_1}\sum_{J_2\subset I_2}N^{17+\frac23}\|a_n\|^{9}_{\ell^{9}(J_1)}\|a_n\|^{9}_{\ell^{9}(J_2)}\\&\lesssim_\epsilon N^{20+\epsilon}\|a_n\|^{9}_{\ell^{9}(I_1)}\|a_n\|^{9}_{\ell^{9}(I_2)}.
\end{align*}	
\end{proof}	
We next use Theorem \ref{b13} to prove a bilinear counterpart of Theorem \ref{c7}. Note that we are renaming the variables $x_3,x_4,x_5$, only for convenience.

\begin{prop}
\label{c22}	
	Let $\Omega=[0,N^3]^4\times[0,N]$. Let $K\gg 1$.
We consider arbitrary integers $N_0,M$ satisfying  $1\le M\le \frac{N_0}{K}$ and $N_0+[M,2M]\subset [\frac{N}{2},N]$.

Let $H_1,H_2$ be intervals of length $\frac{M}K$ inside $N_0+[M,2M]$ such that $\dist(H_1,H_2)\ge \frac{M}{K}$. Then
$$\int_{\Omega}\prod_{i=1}^2|\sum_{n\in H_i}a_ne(\frac{n}{N}x_1+(\frac{n}{N})^2x_2+(\frac{n}{N})^4x_3+(\frac{n}{N})^5x_4+(\frac{n}{N})^3x_5)|^9dx\lesssim_\epsilon M^{2+\epsilon}N^{18}\|a\|^{18}_{\ell^9(N_0+[M,2M])}.$$	
\end{prop}
\begin{proof}
Write $H_1=N_0+I_1$, $H_2=N_0+I_2$ with $I_1,I_2$ intervals of length $\frac{M}K$ inside $[M,2M]$ and with separation $\ge \frac{M}{K}$. Let
$${\phi}_3(t)=t^3+\frac{M}{4N_0}t^4,\;\;\phi_4(t)=t^4+\frac{2M}{5N_0}t^5,\;\;\phi_5(t)=t^5$$
and
$$\begin{cases}y_1=\frac{M}{N}(x_1+A_2x_2+A_3x_3+A_4x_4+A_5x_5)\\y_2=(\frac{M}{N})^2(x_2+B_3x_3+B_4x_4+B_5x_5)\\y_3=(\frac{M}{N})^3(\frac{4N_0}{N}x_3+\frac{10N_0^2}{N^2}x_4+x_5)\\y_4=(\frac{M}{N})^4(\frac{5N_0}{2N}x_4-\frac{N}{4N_0}x_5)\\y_5=(\frac{M}{N})^5\frac{N^2}{10N_0^2}x_5.
\end{cases}$$
The linear transformation has Jacobian equal to $(\frac{M}{N})^{15}$, and maps $\Omega$ to a subset of $$\tilde{\Omega}=[-O(MN^2),O(MN^2)]\times [-O(M^2N),O(M^2N)]\times [-O(M^3),O(M^3)]^2\times [-O(M),O(M)].$$
The reader may check that (with the right choice of $A_2,\ldots,B_5=O(1)$) we have
$$|\sum_{n\in H_i}a_ne(\frac{n}{N}x_1+(\frac{n}{N})^2x_2+(\frac{n}{N})^4x_3+(\frac{n}{N})^5x_4+(\frac{n}{N})^3x_5)|=$$$$|\sum_{m_i\in I_i}a_{m_i+N_0}e(\frac{m_i}{M}y_1+(\frac{m_i}{M})^2y_2+\phi_3(\frac{m_i}{M})y_3+\phi_4(\frac{m_i}{M})y_4+\phi_5(\frac{m_i}{M})y_5)|.$$
We invoke periodicity in $y_1$ and $y_2$ together with Theorem \ref{b13} to conclude the proof.

\end{proof}

Theorem \ref{c7} follows from Proposition \ref{c22} via an argument identical to the one at the end of the previous section. Details are left to the reader.

\section{Mean value estimates for sums along the paraboloid}\label{sec:L2paraboloid}

This section is devoted to proving the following theorem.

\begin{thm}\label{thm:L2paraboloid}
	For $d\ge 2$ let $\sigma$ be a measure on $\T^d$ such that the Fourier decay \eqref{eq:decbeta} holds with $\beta=\frac{d-1}{2}$.
	Then the estimate
	\begin{equation}\label{eq:L2paraboloid}
	\int_{\T^d} |S^{\mathbb{P}}_{a, d}(x, N)|^2\,d\sigma(x)\lesssim \|a\|_{\ell^2(\{1,\ldots,N\}^{d-1})}^2\, N^{(d-3)/2} \begin{cases}
	N^{1/2},&\qquad d=2,\\
	\log N,&\qquad d=3,\\
	1,&\qquad d\ge 4,
	\end{cases}
	\end{equation}
	holds with the implicit constant independent of $N$.
	
	Moreover, the above estimate is sharp, in the sense that the constant on the right-hand side cannot be improved in terms of the dependence on $N$.
\end{thm}
The above result shows that  Conjecture \ref{a42} holds in the case $d=2$ and also, up to the $\log N$ loss, in the case $d=3$.
\begin{proof}
	Without the loss of generality assume that $\|a\|_{\ell^2}=1$.
	Using \eqref{eq:decbeta} we get
	
	\begin{align*}
\int_{\T^d} |S^{\mathbb{P}}_{a, d}(x, N)|^2\,d\sigma(x)&=\sum\limits_{\substack{\m,\n\in\{1,\ldots,N\}^{d-1}  }}a_{\m}\,\overline{a_{\n}}\,\widehat{\sigma}(\m-\n, |\m|^2-|\n|^2)
	\\
	&\lesssim \sum\limits_{\substack{\m,\n\in\{1,\ldots,N\}^{d-1} }}|a_{\m}\,a_{\n}|\,(1+|(\m-\n, |\m|^2-|\n|^2)|)^{-\frac{d-1}{2}}
	\\
	&\lesssim 1+\sum\limits_{\substack{\m,\n\in\{1,\ldots,N\}^{d-1} \\ \m\neq \n }}|a_{\m}\,a_{\n}| \, ||\m|^2-|\n|^2|^{-\frac{d-1}{2}}.
	\end{align*}
	Let
	\begin{align*}
	c_{\m,\n}=
	\begin{cases}
	||\m|^2-|\n|^2|^{-\frac{d-1}{2}},&\qquad \m\neq \n\\
	0,&\qquad \m=\n
	\end{cases}.
	\end{align*}
	To simplify notation let us assume that from now on $a_{\n}\ge 0$ for all $\n$.
	Let $$\mathcal{S}=\{(\m,\n)\in\{1,\ldots,N\}^{d-1} \times \{1,\ldots,N\}^{d-1} :|\m|\in(|\n|/2, 2|\n|)\}.$$ Then we have
	\begin{align}\label{eq:cmn}
	c_{\m,\n}\simeq \begin{cases}
	|\m|^{-\frac{d-1}{2}}\,||\m|-|\n||^{-\frac{d-1}{2}},&\qquad (\m,\n)\in\mathcal{S},\quad \m\neq \n\\
	\left(\max\{|\m|,|\n|\}\right)^{-(d-1)},&\qquad (\m,\n)\notin\mathcal{S}\\
	0,&\qquad \m=\n
	\end{cases}.
	\end{align}
	We begin with analyzing the contribution from $(\m,\n)\notin \mathcal{S}$. In view of \eqref{eq:cmn} we have (the ranges of summation in $\m$ and $\n$ are $\{1,\ldots,N\}^{d-1}$ unless indicated otherwise)
	\begin{align*}
	\sum\limits_{(\m,\n)\notin\mathcal{S}}a_{\m}\,a_{\n}\, c_{\m,\n}\lesssim\sum\limits_{\m, \n}a_{\m}\,a_{\n}\, \left(\max\{|\m|,|\n|\}\right)^{-(d-1)}:=D.
	\end{align*}
	We shall prove that
	$$
	D\lesssim\sqrt{D},
	$$
	which immediately implies $D\lesssim 1$, and consequently
	\begin{equation}\label{eq:notinS}
	\sum\limits_{(\m,\n)\notin\mathcal{S}}a_{\m}\,a_{\n}\, c_{\m,\n}\lesssim 1.
	\end{equation}
	By symmetry it suffices to show
	$$
	\sum\limits_{\substack{\m, \n\\ |\m|\ge|\n|}}a_{\m}\,a_{\n}\, |\m|^{-(d-1)}\lesssim\sqrt{D}.
	$$
	Using the Cauchy--Schwarz inequality we obtain
	\begin{align*}
	\sum\limits_{\substack{\m, \n\\ |\m|\ge |\n|}}a_{\m}\,a_{\n}\, |\m|^{-(d-1)}&=\sum_{\m}a_{\m}\, |\m|^{-(d-1)}\,\sum\limits_{\substack{\n:\; |\n|\le |\m|}}\,a_{\n}\le
	\left(\sum_{\m} |\m|^{-2(d-1)}\,\Big(\sum\limits_{\substack{\n:\; |\n|\le |\m|}}\,a_{\n}\Big)^2\right)^{1/2}
	\\
	&\le(\sum_{\n, \n'}a_{\n} \, a_{\n'}\, \sum\limits_{\substack{\m:\; |\m|>|\n|, |\n'|}}\,|\m|^{-2(d-1)})^{1/2}
	\\
	&\simeq
	(\sum_{\n, \n'}a_{\n} \, a_{\n'}\, \left(\max\{|\n|,|\n'|\}\right)^{-(d-1)})^{1/2}=\sqrt{D}.
	\end{align*}
	It remains to treat the contribution from $(\m,\n)\in \mathcal{S}$. Applying Cauchy-Schwarz inequality we get
	\begin{align*}
	\sum\limits_{(\m,\n)\in\mathcal{S}}a_{\m}\,a_{\n}\, c_{\m,\n}
	&\le
	(\sum\limits_{(\m,\n)\in\mathcal{S}}a_{\m}^2\, c_{\m,\n})^{1/2}
	(\sum\limits_{(\m,\n)\in\mathcal{S}}a_{\n}^2\, c_{\m,\n})^{1/2}
	\\
	&\le
	\sum\limits_{(\m,\n)\in\mathcal{S}}a_{\m}^2\, c_{\m,\n}
	+
	\sum\limits_{(\m,\n)\in\mathcal{S}}a_{\n}^2\, c_{\m,\n}.
	\end{align*}
	By symmetry it suffices to estimate the first sum above. Recalling that $\|a\|_{\ell^2}=1$ we obtain
	\begin{align*}
	\sum\limits_{(\m,\n)\in\mathcal{S}}a_{\m}^2\, c_{\m,\n}=\sum_\m a_{\m}^2 \sum\limits_{\n: (\m,\n)\in\mathcal{S}}c_{\m,\n}
	\le \sup_{\m} \sum\limits_{\n: (\m,\n)\in\mathcal{S}}c_{\m,\n}.
	\end{align*}
	We need to prove that
	\begin{equation}\label{eq:S}
	\sup_{\m} \sum\limits_{\n: (\m,\n)\in\mathcal{S}}c_{\m,\n}\lesssim
	N^{(d-3)/2}
	\begin{cases}
	N^{1/2},&\qquad d=2\\
	\log N,&\qquad d=3\\
	1,&\qquad d\ge 4
	\end{cases}.
	\end{equation}
	To this end fix $m\in \{1,\ldots,N\}^{d-1}$. Using \eqref{eq:cmn} we get
	\begin{align*}
	\sum\limits_{\n: \;(\m,\n)\in\mathcal{S}}c_{\m,\n}&\simeq |\m|^{-\frac{d-1}{2}}\sum\limits_{\substack{\n: |\m|/2<|\n|<2|\m|\\ \n\neq \m}}\left||\n|-|\m|\right|^{-\frac{d-1}{2}}
	\\
	&\simeq |\m|^{-\frac{d-1}{2}}\sum\limits_{\substack{k\in \left[|\m|/2,2|\m|\right]\cap \Z\\ k\neq |\m|}}\left|k-|\m|\right|^{-\frac{d-1}{2}}\left|\{\n: |\n|\in [k,k+1)\}\right|
	\\
	&\simeq |\m|^{-\frac{d-1}{2}}\sum\limits_{\substack{k\in \left[|\m|/2,2|\m|\right]\cap \Z\\ k\neq |\m|}}\left|k-|\m|\right|^{-\frac{d-1}{2}}k^{d-2}
	\\
	&\simeq |\m|^{-\frac{d-1}{2}}|\m|^{d-2}\sum\limits_{\substack{k\in \left[|\m|/2,2|\m|\right]\cap \Z\\ k\neq |\m|}}\left|k-|\m|\right|^{-\frac{d-1}{2}}
	\\
	&\simeq |\m|^{-\frac{d-1}{2}}|\m|^{d-2}\sum\limits_{k=1}^{|\m|}k^{-\frac{d-1}{2}}
	\simeq |\m|^{-\frac{d-1}{2}}|\m|^{d-2}
	\begin{cases}
	|\m|^{-\frac{d-1}{2}+1},&\qquad d=2\\
	\log |\m|,&\qquad d=3\\
	1,&\qquad d\ge 4
	\end{cases}.
	\end{align*}
	We have thus showed
	\begin{equation*}
	\sum\limits_{\n: (\m,\n)\in\mathcal{S}}c_{\m,\n}
	\simeq
	|\m|^{\frac{d-3}{2}}
	\begin{cases}
	|\m|^{-\frac{d-1}{2}+1}&\qquad d=2,\\
	\log |\m|,&\qquad d=3\\
	1,&\qquad d\ge 4,
	\end{cases},
	\end{equation*}
	and \eqref{eq:S} follows.
	Combining \eqref{eq:notinS} and \eqref{eq:S} gives the first part of the theorem.
	
	We now prove sharpness in the most interesting case, $d=3$. A similar argument works for the other values of $d$. Using  $a_{\n}\equiv1$,  it suffices to verify the following lower bound
	\begin{equation}
	\sum\limits_{\substack{\m\in\{1,\ldots,N\}^{2} \\ \n\in\{1,\ldots,N\}^{2} }}(1+|(\m-\n, |\m|^2-|\n|^2)|)^{-1}\gtrsim N^2 \log N.
	\end{equation}
	Note that  $|\m|^2-|\n|^2=(|\m|-|\n|)(|\m|+|\n|)>|\m|-|\n|$, thus
	$$
	1+|(\m-\n, |\m|^2-|\n|^2)|\simeq 1+||\m|^2-|\n|^2|.
	$$
	Consequently, we can estimate
	\begin{align*}
	\sum\limits_{\substack{\m,\n\in\{1,\ldots,N\}^{2}}}(1+|(\m-\n, |\m|^2-|\n|^2)|)^{-1}
	&\gtrsim
	\sum\limits_{\substack{\m,\n\in\{1,\ldots,N\}^{2}\\|\m|>|\n|+10 }}(|\m|^2-|\n|^2)^{-1}
	\\
	&\simeq \sum_{l=1}^{O(N)} \sum_{k=1}^{l-10}\frac{|\{\n:|\n|\in[l, l+1)\}| |\{\m:|\m|\in[k, k+1)\}|}{l^2-k^2}
	\\
	&\simeq \sum_{l=1}^{O(N)} \sum_{k=1}^{l-10}\frac{k\, l}{l(l-k)}\simeq \sum_{l=1}^{O(N)} \sum_{k=1}^{l-10}\left(\frac{k-l}{l-k}+\frac{l}{l-k}\right)
	\\
	&\simeq \sum_{l=1}^{O(N)} (l\log l-l)\simeq N^2\log N.
	\end{align*}
\end{proof}

\section{$\epsilon$-free $L^4$ estimate for $d=3$}
\label{four}

In this section we verify the following scale-independent estimate.

\begin{thm}\label{thm:L4}
	Let $\sigma$ be a measure on $\T^3$ such that the estimate \eqref{eq:decbeta} holds with some $\beta>2/3$.
	Then the estimate
	\begin{equation}\label{eq:L4}
	\int_{\T^3} |S_{a, 3}(x, N)|^4\,d\sigma(x)\lesssim \|a\|_{\ell^2}^4
	\end{equation}
	holds with the implicit constant  independent of $N$.
\end{thm}
There does not seem to be an approach to this theorem using Proposition \ref{a27}, not even when $\beta=1$. Given $\beta>2/3$, \eqref{eq:L4} would follow from the  scale-independent  estimate
\begin{equation}
\label{a50}
\|\sum_{n=1}^Na_ne(n^3t)\|_{L^p([0,1])}\lesssim \|a\|_{\ell^2},
\end{equation}
conjectured to hold for $p<6$. Indeed, using this with $p>\frac{4}{\beta}$ we find
\begin{align*}
\int_{\T^3} |S_{a, 3}(x, N)|^4d\sigma(x)&\lesssim \sum_{0\le j\lesssim \log N}2^{(3-\beta)j}\int_{[-2^{-j},2^{-j}]^3}|S_{a, 3}(x, N)|^4dx\\&\le \sum_{0\le j\lesssim \log N}2^{(1-\beta)j}\sup_{b:\;|b_n|=|a_n|}\int_{[-2^{-j},2^{-j}]}|\sum_{n=1}^{N}b_ne(n^3x_3)|^4dx_3 \\&\le \sum_{0\le j\lesssim \log N}2^{(\frac4p-\beta)j}\sup_{b:\;|b_n|=|a_n|}(\int_0^1|\sum_{n=1}^{N}b_ne(n^3x_3)|^pdx_3)^{4/p}\\&\lesssim \|a\|_{\ell^2}^4.
\end{align*}
However, \eqref{a50} is not known even for $p=4$ and $a_n=1$. The difficulty of this inequality is already captured by the following result in \cite{Si}: there is a sequence $s_k\to\infty$ such that the equation $n^3+m^3=s_k$ has $(\log s_k)^{3/5}$ integral solutions.
\smallskip

Our proof of \eqref{eq:L4} will be elementary, and will involve delicate counting arguments. We never use the variable $x_1$ in our proof. In fact, an inspection of the argument reveals that we prove the stronger estimate
\begin{equation}\label{eq:L4dim2}
\int_{\T^2} |S_{a, 3}((0,x_2,x_3), N)|^4\,d\sigma(x_2, x_3)\lesssim \|a\|_{\ell^2}^4,
\end{equation}
for measure $\sigma$ on $\T^2$,  subject to only the decay condition
$$|\widehat{d\sigma}(\xi)|\lesssim (1+|\xi|)^{-\beta},\;\beta>2/3.$$
This rate of decay  is sharp,  in the sense that the estimate \eqref{eq:L4dim2} fails for $\beta<2/3$. To see that, let $\sigma$ be a measure on $\T^2$ such that $\widehat{\sigma}$ is real, positive and $\widehat{\sigma}(\xi)\gtrsim \frac{1}{(1+|\xi|)^\beta}$. Let $a$ be the sequence given by
	$
	a_n=\mathbbm{1}_{[1,2^{j/3}]}
	$
	for some fixed $2^{j/3}\lesssim N$. Then $\|a\|_{\ell^2}^4\simeq 2^{2j/3}$ and
	\begin{align*}
	\int_{\T^2} |S_{a, 3}((0,x_2,x_3), N)|^4\,d\sigma(x_2, x_3)&\gtrsim \frac{1}{2^{j\beta}} \left|\bigg\{(n_1,\dots, n_4)\in [1, 2^{\frac{j}3}]^4: \begin{cases}|n^2_1+n^2_2-n^2_3-n^2_4|\lesssim  2^j \\ |n^3_1+n^3_2-n^3_3-n^3_4| \lesssim 2^j \end{cases} \bigg\}\right|\\&\gtrsim \frac{1}{2^{j\beta}} (2^{j/3})^4\gg\|a\|_{\ell^2}^4,
	\end{align*}
	if $\beta<2/3$.
\smallskip

Before we present the proof of Theorem \ref{thm:L4} we need some technical preparation.
For $C>0$ define
\begin{equation}\label{eq:FC}
F_C(a):=\sqrt{a(C-a^3)},\qquad a\in(0,C^{1/3}).
\end{equation}
The first lemma is just a simple estimate of the size of increments $F_C$.
\begin{lem}\label{lem:calc}
	Let $C>1000$ be a fixed constant. Then
	\begin{itemize}
		\item[(a)] $F_C$ attains its maximum on $(0, C^{1/3})$ at $a_{max}=(C/4)^{1/3}$ and $F_C'(x)>0$ for $x\in (0, a_{max})$, and $F_C'(x)<0$ for $x\in (a_{max}, C^{1/3})$.
		\item[(b)] For each $y\in [0,a_{max}]$ one has
		$$
		|F_C(a_{max}-y)-F_C(a_{max})|\simeq y^2,
		$$
		with the implicit constant independent of $C$.
	\end{itemize}
\end{lem}

\begin{proof}
	Part (a) is straightforward. To show (b), we begin with fixing $y\in [0,a_{max}]$.  Since $F_C(a_{max}-y)\le F_C(a_{max})$,  $F_C(a_{max})\simeq C^{2/3}$ and $a_{max}=(C/4)^{1/3}$, we have
	\begin{align*}
	|F_C(a_{max}-y)-F_C(a_{max})|&\simeq \left|\frac{F_C^2(a_{max}-y)-F_C^2(a_{max})}{F_C(a_{max})}\right|\\
	&\simeq C^{-2/3}\left|(a_{max}-y)(C-(a_{max}-y)^3)-a_{max}(C-a_{max}^3)\right|
	\\
	&=C^{-2/3}\left|-yC+4a_{max}^3y-6a_{max}^2 y^2+4a_{max}y^3-y^4\right|
	\\
	&=C^{-2/3}\left|-6a_{max}^2 y^2+4a_{max}y^3-y^4\right|\\
	&\simeq C^{-2/3}a_{max}^2 y^2\simeq y^2,
	\end{align*}
	so (b) is proved.
\end{proof}

\begin{lem}\label{lem:convex}
	For $M\in\N$ let $f:[0,M]\rightarrow \R_+$ be such that $f'>0$, $f''<0$ and $f(0)=0$. Then for any $\alpha>0$ the following estimate holds:
	\begin{align*}
	\sup_{x\in\R_+}\sum_{j=1}^M\frac{1}{|f(j)-x|^\alpha+1}\le 2+4\sum_{j=1}^M\frac{1}{(f(M)-f(j))^\alpha+1}.
	\end{align*}
\end{lem}
The above lemma asserts that the supremum of the sum on the left-hand side is essentially comparable with its value at $x=f(M)$.

\begin{proof}
	Fix $x\in\R_+$. If $x\ge f(M)$ then
	for each $j$ we have
	$$
	\frac{1}{(x-f(j))^\alpha+1}\le \frac{1}{(f(M)-f(j))^\alpha+1},
	$$
	and consequently
	\begin{align*}
	\sum_{j=1}^M\frac{1}{(x-f(j))^\alpha+1}\le \sum_{j=1}^M\frac{1}{(f(M)-f(j))^\alpha+1}.
	\end{align*}
	Therefore from now on we assume that $x\in[0, f(M)]$. If $x$ is not in the range of $f$, i.e.\ $x\notin \{f(0), \ldots, f(M)\}$, let $k\in \{0, 1,\ldots, M-1\}$ be such that $x\in (f(k), f(k+1))$. Note that then for any $j\ge k+2$ we have
	\begin{align*}
	|f(j)-x|^\alpha=(f(j)-x)^\alpha>(f(j)-f(k+1))^\alpha=|f(j)-f(k+1)|^\alpha
	\end{align*}
	and for any $j\le k-1$
	\begin{align*}
	|f(j)-x|^\alpha=(x-f(j))^\alpha>(f(k)-f(j))^\alpha=|f(j)-f(k)|^\alpha.
	\end{align*}
	Using the above relations for $j\in\{1,\ldots,k-1\}\cup\{k+2,\ldots, M\}$ and estimating the terms corresponding to $j\in\{k,k+1\}$ trivially by $1$ we obtain, with the convention that the sum over an empty set of indices is zero,
	\begin{align*}
	\sum_{j=1}^M\frac{1}{|f(j)-x|^\alpha+1}&\le 2+\sum_{j=1}^{k-1}\frac{1}{|f(j)-f(k)|^\alpha+1}+\sum_{j=k+2}^{M}\frac{1}{|f(j)-f(k+1)|^\alpha+1}\\
	&\le 2+2\max_{k\in\{0,1,\ldots,M\}}\sum_{j=1}^M \frac{1}{|f(j)-f(k)|^\alpha+1}.
	\end{align*}
	Consequently, it remains to prove that
	\begin{equation}\label{eq:max}
	\max_{k\in\{0, 1,\ldots,M\}}\sum_{j=1}^M \frac{1}{|f(j)-f(k)|^\alpha+1}\le 2\sum_{j=1}^M\frac{1}{(f(M)-f(j))^\alpha+1}.
	\end{equation}
	To this end fix $k\in\{0,1, \ldots, M\}$. The key observation is that for a fixed $l\in\N$ the distances between pairs of values of
	$f$ at arguments separated by $l$ decrease. Indeed, using the condition $f''<0$ we can estimate for $M_1\le M_2$
	$$
	f(M_1+l)-f(M_1)=\int_{M_1}^{M_1+l}f'(x)\,dx\ge \int_{M_2}^{M_2+l}f'(x)\,dx=f(M_2+l)-f(M_2).
	$$
	In view of the above we get
	\begin{align*}
	\begin{cases}
	|f(j)-f(k)|^\alpha\ge |f(M)-f(M-(k-j))|^\alpha,\qquad \textrm{for}\quad 1\le j\le k,\\
	|f(j)-f(k)|^\alpha\ge |f(M)-f(M-(j-k))|^\alpha,\qquad \textrm{for}\quad k+1\le j\le M.
	\end{cases}
	\end{align*}
	Using the above relations we get
	\begin{align*}
	\sum_{j=1}^M \frac{1}{|f(j)-f(k)|^\alpha+1}&\le \sum_{j=M-k}^M \frac{1}{(f(M)-f(j))^\alpha+1}+\sum_{j=k}^M \frac{1}{(f(M)-f(j))^\alpha+1}
	\\
	&\le 2\sum_{j=1}^M \frac{1}{(f(M)-f(j))^\alpha+1},
	\end{align*}
	which completes the proof of \eqref{eq:max} and consequently the entire lemma.
\end{proof}

\begin{cor}\label{cor:cip}
	Let $C>1000, D\ge 1$ and let $F_C:(0,C^{1/3})\rightarrow \R_+$ be the function defined by \eqref{eq:FC}, with $a_{max}=(C/4)^{1/3}$. Then the following estimate holds for each $\beta>1/2$
	\begin{align*}
	\sup_{x\in\R_+}\sum_{a=1}^{\lfloor a_{max}\rfloor}\frac{1}{|F_C(a)-x|^\beta+D^\beta}\lesssim D^{1/2-\beta},
	\end{align*}
	with the implicit constant independent of $C$ and $D$.
\end{cor}

\begin{proof}
	Note that the estimate we need to prove is equivalent to
	\begin{align*}
	\sup_{x\in\R_+}\sum_{a=1}^{\lfloor a_{max}\rfloor}\frac{1}{|F_C(a)/D-x|^\beta+1}\lesssim D^{1/2}.
	\end{align*}
	Applying Lemma \ref{lem:convex} with $f=F_C/D$ and $M=\lfloor a_{max}\rfloor$ we get
	\begin{align*}
	\sup_{x\in\R_+}\sum_{a=1}^{\lfloor a_{max}\rfloor}\frac{1}{|F_C(a)/D-x|^\beta+1}&\lesssim \sum_{a=1}^{\lfloor a_{max}\rfloor}\frac{1}{|F_C(a)/D-F_C(\lfloor a_{max}\rfloor)/D|^\beta+1}
	\\&
	=\sum_{n=1}^{\lfloor a_{max}\rfloor}\frac{1}{D^{-\beta}|F_C(\lfloor a_{max}\rfloor-n)-F_C(\lfloor a_{max}\rfloor)|^\beta+1}:=S.
	\end{align*}
	%Writing $\lfloor a_{max}\rfloor=a_{max}-\epsilon$, with $\epsilon\in[0,1)$, and
	Due to monotonicity of $F_C$ on $(0, a_{max})$ we can estimate
	\begin{align*}
	|F_C(\lfloor a_{max}\rfloor-n)-F_C(\lfloor a_{max}\rfloor)|%&=|F(a_{max}-\epsilon-n)-F(a_{max}-\epsilon)|
	%\\
	%&\ge |F(a_{max}-\epsilon-n)-F(a_{max})|-|F(a_{max})-F(a_{max}-\epsilon)|
	%\\
	&\ge |F_C(a_{max}-n)-F_C(a_{max})|-|F_C(a_{max})-F_C(a_{max}-1)|.
	\end{align*}
	Now using Lemma \ref{lem:calc} we get
	$$
	|F_C(a_{max}-n)-F_C(a_{max})|\simeq n^2
	$$
	and
	$$
	|F_C(a_{max})-F_C(a_{max}-1)|\simeq 1.
	$$
	Therefore we can find and absolute constant $n_0\in\N$ such that for each $n\ge n_0$ we have
	$$
	|F_C(a_{max}-n)-F_C(a_{max})|\ge 2|F_C(a_{max})-F_C(a_{max}-1)|.
	$$
	Then we have for $n\ge n_0$
	\begin{align*}
	|F_C(\lfloor a_{max}\rfloor-n)-F_C(\lfloor a_{max}\rfloor)|\simeq |F_C(a_{max}-n)-F_C(a_{max})|\simeq n^2.
	\end{align*}
	Thus we can estimate
	\begin{align*}
	S\lesssim \sum_{n=1}^{n_0} 1+\sum_{n=n_0}^{\lfloor a_{max}\rfloor}\frac{1}{D^{-\beta}n^{2\beta}+1}\lesssim
	\sum_{n=1}^{D^{1/2}}1+D^\beta\sum_{n=D^{1/2}}^{\infty} \frac{1}{n^{2\beta}}\simeq D^{1/2},
	\end{align*}
	which concludes the proof of the corollary.
\end{proof}

Now we are ready to prove Theorem \ref{thm:L4}.
\begin{proof}[Proof of Theorem \ref{thm:L4}]
	
	Using \eqref{eq:decbeta} we get
	\begin{align*}
	\int_{\T^3} |S_{a, 3}(x, N)|^4\,d\sigma(x)&=\sum_{n_1, n_2,n_3, n_4=1}^N  a_{n_1}a_{n_2}\overline{a_{n_3}}\overline{a_{n_4}}
	\\
	&\quad\times\widehat{\sigma}(n_1+n_2-n_3-n_4, n^2_1+n^2_2-n^2_3-n^2_4, n^3_1+n^3_2-n^3_3-n^3_4)
	\\
	&\lesssim\sum_{n_1, n_2,n_3, n_4=1}^N  |a_{n_1}a_{n_2}a_{n_3}a_{n_4}|
	\\
	&\quad\times\frac{1}{(|(n^2_1+n^2_2-n^2_3-n^2_4, n^3_1+n^3_2-n^3_3-n^3_4)|+1)^\beta}
	\\
	&:=\sum_{n_1, n_2,n_3, n_4=1}^N  |a_{n_1}a_{n_2}a_{n_3}a_{n_4}|c_{n_1, n_2,n_3, n_4}.
	\end{align*}
	Note that
	\begin{align*}
	\sum\limits_{\substack{n_1, n_2,n_3, n_4\\ n_2=n_4}}^N  |a_{n_1}a_{n_2}a_{n_3}a_{n_4}|\,c_{n_1, n_2,n_3, n_4}
	&=
	\sum_{n_1, n_2,n_3=1}^N  |a_{n_1}a_{n_2}^2 a_{n_3}|\, c_{n_1, n_2,n_3, n_2}
	\\
	&=
	\|a\|_{\ell^2}^2\sum_{n_1, n_3=1}^N  |a_{n_1}\, a_{n_3}| \, \frac{1}{(|(n^2_1-n^2_3, n^3_1-n^3_3)|+1)^\beta}
	\\
	&\le
	\|a\|_{\ell^2}^2\sum_{n_1, n_3=1}^N  |a_{n_1}\, a_{n_3}| \, \frac{1}{(|n^2_1-n^2_3|+1)^{1/2}}
	\\
	&\lesssim
	\|a\|_{\ell^2}^4.
	\end{align*}
	By symmetry, we also have
	\begin{align*}
	\sum\limits_{\substack{n_1, n_2,n_3, n_4\\ n_1=n_3}}^N  |a_{n_1}a_{n_2}a_{n_3}a_{n_4}|\,c_{n_1, n_2,n_3, n_4}
	\lesssim
	\|a\|_{\ell^2}^4.
	\end{align*}
	Therefore it remains to show
	\begin{align*}
	\sum\limits_{\substack{n_1, n_2,n_3, n_4=1\\ n_1\neq n_3\\n_2\neq n_4}}^N  |a_{n_1}a_{n_2}a_{n_3}a_{n_4}|\,c_{n_1, n_2,n_3, n_4}
	\lesssim
	\|a\|_{\ell^2}^4.
	\end{align*}
	
	Applying the Cauchy--Schwarz inequality and using the symmetry $c_{n_1, n_2, n_3, n_4}=c_{n_3, n_4, n_1, n_2}$ we get
	\begin{align*}
	&\sum\limits_{\substack{n_1, n_2,n_3, n_4=1\\ n_1\neq n_3\\n_2\neq n_4}}^N  |a_{n_1}a_{n_2}a_{n_3}a_{n_4}|c_{n_1, n_2,n_3, n_4}\\
	&\le (\sum\limits_{\substack{n_1, n_2,n_3, n_4=1\\ n_2\neq n_4}}^N |a_{n_1}|^2\, |a_{n_3}|^2\, c_{n_1, n_2,n_3, n_4})^{1/2} (\sum\limits_{\substack{n_1, n_2,n_3, n_4=1\\ n_1\neq n_3}}^N |a_{n_2}|^2\, |a_{n_4}|^2\, c_{n_1, n_2,n_3, n_4})^{1/2}
	\\
	&\le \|a\|_{\ell^2}^4(\sup_{n_1, n_3\in\N}\sum\limits_{\substack{n_2, n_4\in\N\\ n_2\neq n_4}} c_{n_1, n_2,n_3, n_4})^{1/2} (\sup_{n_2, n_4\in\N}\sum\limits_{\substack{n_1, n_3\in\N\\ n_1\neq n_3}}  c_{n_1, n_2,n_3, n_4})^{1/2}
	\\
	&=\|a\|_{\ell^2}^4\sup_{n_1, n_3\in\N}\sum\limits_{\substack{n_2, n_4\in\N\\ n_2\neq n_4}} c_{n_1, n_2,n_3, n_4}.
	\end{align*}
	Therefore it remains to prove that
	\begin{align*}
	\sup_{n_1, n_3\in\N}\sum\limits_{\substack{n_2, n_4\in\N\\ n_2\neq n_4}} c_{n_1, n_2,n_3, n_4}\lesssim 1.
	\end{align*}
	Calling $C=n_3^3-n_1^3$, $D=n_3^2-n_1^2$, $a=n_2-n_4\neq 0$ and $b=n_2+n_4$, we reduce the problem to showing
	\begin{align*}
	\sup_{C, D\in\N}S(C,D):=\sup_{C, D\in\N}\sum\limits_{\substack{a,b\in\N \\b\ge a>0}} \frac{1}{|ab-D|^\beta+|a(a^2+3b^2)-C|^\beta+1}\lesssim 1.
	\end{align*}
	Fix $C,D\in\N$ and let $j\in \N$ be such that $C\simeq 2^j$. Let $I_k=\{(a,b)\in\N^2: b\ge a>0, |a(3b^2+a^2)-C|\simeq 2^k\}$. We observe that, since $|a(a^2+3b^2)-C|\gtrsim ab^2$ if $(a,b)\in I_k$ with $k\ge j+1$, we have
		$$
	S(C,D)\simeq \sum_{1\le 2^k\le 2^j} \sum_{(a,b)\in I_k} \frac{1}{|ab-D|^\beta+2^{k\beta}+1}+\sum_{b\ge a\ge 1}\frac{1}{ab^2}.
	$$
	Note that the second sum is $O(1)$, and that  $a\lesssim 2^{j/3}$ for each $a$ contributing to the first sum.
	Notice also that for a fixed $a$ there are at most ${O}(\frac{2^k}{\sqrt{a 2^j}}+1)$ choices of $b$ such that $(a,b)\in I_k$. Indeed, if $b_1,b_2\in\N$ are such that $(a,b_1), (a, b_2)\in I_k$ then we have
	$$
	|a(a^2+3b_1^2)-C-(a(a^2+3b_2^2)-C)|=3|a(b_1^2-b_2^2)|\lesssim 2^k,
	$$
	and it follows that
	$$
	|b_1-b_2|\lesssim \frac{2^k}{a (b_1+b_2)}\simeq\frac{2^k}{\sqrt{a}\sqrt{a(b_1^2+b_2^2)}}.
	$$
	Noticing that $ab_i^2\simeq 2^j$ for $i=1,2$ we obtain
	$$
	|b_1-b_2|\lesssim \frac{2^k}{\sqrt{a 2^j}},
	$$
	which implies the desired upper bound for the number of possible choices of $b$.
	
	This motivates a further decomposition
	\begin{align*}
	\sum_{1\le 2^k\le 2^j} \sum_{(a,b)\in I_k} \frac{1}{|ab-D|^\beta+2^{k\beta}+1}&= \sum_{1\le 2^k\le 2^j} \sum_{(a,b)\in I^1_k} \frac{1}{|ab-D|^\beta+2^{k\beta}+1}
	\\
	&\quad+ \sum_{1\le 2^k\le 2^j} \sum_{(a,b)\in I^2_k} \frac{1}{|ab-D|^\beta+2^{k\beta}+1}
	\\
	&=:S_1+S_2,
	\end{align*}
	where
	\begin{align*}
	I^1_k=\{(a,b)\in I_k: \sqrt{a2^j}\le 2^k\},\;\;\;\;I^2_k=\{(a,b)\in I_k: \sqrt{a2^j}> 2^k\}.
	\end{align*}
	Using the fact that for a fixed $a$ there are at most $O(\frac{2^k}{\sqrt{a 2^j}}+1)$ choices of $b$ such that $(a,b)\in I_k$ we can estimate $S_1$ as follows
	\begin{align*}
	S_1\le  \sum_{1\le 2^k\le 2^j} \sum_{(a,b)\in I^1_k} \frac{1}{2^{k\beta}}\lesssim  \sum_{1\le 2^k\le 2^j} \frac{1}{2^{k\beta}}|I_k^1|\lesssim
	\sum_{1\le 2^k\le 2^j} \frac{1}{2^{k\beta}}\sum_{1\le a\lesssim 2^{j/3}}\frac{2^k}{\sqrt{a 2^j}}\lesssim
	\sum_{1\le 2^k\le 2^j} \frac{1}{2^{k\beta}}\frac{2^k}{2^{j/3}}\lesssim 1,
	\end{align*}
	for $\beta>2/3$.
	
	We pass to the analysis of $S_2$. Notice that given $(a,b)\in I_k^2$ there cannot exist $b_2\neq b$ such that $(a, b_2)\in I_k^2$; in other words, each $a$ is associated with at most one $b$.
	
	For a fixed $(a,b)\in I_k^2$ denote $C_k':=a(3b^2+a^2)$. Clearly $C_k'$ depends on $a$ and $b$. Next, let $C_k:=C+100 \times 2^k$. Then  $C_k'\simeq C_k\simeq C$ and $|C_k-C_k'|\lesssim 2^k$. We have
	$$
	ab=\frac{1}{\sqrt{3}}\sqrt{(C_k'-a^3)a}=\frac{1}{\sqrt{3}}F_{C_k'}(a),
	$$
	where for $K>0$ the function $F_K$ is given by \eqref{eq:FC}.
	Furthermore, since $a\le b$, we can estimate
	\begin{equation}\label{eq:amax}
	a\le\sqrt[3]{\frac{a(3b^2+a^2)}{4}}\le\sqrt[3]{\frac{C_k}{4}}=: a_{max}(k).
	\end{equation}
	Since $C_k'$  depends on $(a,b)$ we would like to replace it with $C_k$. We have
	\begin{align*}
	|ab-\frac{1}{\sqrt{3}}\sqrt{(C_k-a^3)a}|=\frac{1}{\sqrt{3}}|\sqrt{(C_k'-a^3)a}-\sqrt{(C_k-a^3)a}|\lesssim \frac{2^k a }{\sqrt{2^j a}}=o(2^k),
	\end{align*}
	due to $a\lesssim 2^{j/3}$. It follows that
	$$
	|ab-D|\ge |\sqrt{(C_k-a^3)a}-D|-o(2^k).
	$$
	Thus
	$$
	|ab-D|^\beta+2^{k\beta}\gtrsim |\sqrt{(C_k-a^3)a}-D|^\beta+2^{k\beta}=|F_{C_k}(a)-D|^\beta+2^{k\beta}.
	$$
	Consequently
	\begin{align*}
	S_2=\sum_{1\le 2^k\le 2^j} \sum_{(a,b)\in I^2_k} \frac{1}{|ab-D|^\beta+2^{k\beta}+1}\lesssim \sum_{1\le 2^k\le 2^j} \sum_{1\le a\le a_{max}(k)} \frac{1}{|F_{C_k}(a)-D|^\beta+2^{k\beta}}.
	\end{align*}
	Finally, using Corollary \ref{cor:cip} with $C=C_k$, for each $k$ such that $1\le 2^k\le 2^j$, we get
	$$
	S_2\lesssim 1,
	$$
	so the theorem is now proved.
\end{proof}

\section{The circle}\label{sec:sphere}
For $N\in N$, let $\mathbb{S}_+^1(N)$ be the upper semicircle of radius $\sqrt{N}$ and denote
$$
S_N:=\mathbb{S}_+^1(N)\cap \Z^2.
$$
It is known that for any $\epsilon>0$ we have
$
|S_N|\lesssim_\epsilon N^\epsilon.
$
Therefore one can estimate using the Cauchy--Schwarz inequality, for any $p\ge 1$ and any finite $d\sigma$
\begin{equation*}
\int_{\T^2} \left|\sum_{\n\in S_N}a_{\n}\, e(\n\cdot x)\right|^pd\sigma(x)\lesssim \left\|\sum_{\n\in S_N}a_{\n}\, e(\n\cdot x)\right\|_\infty^p\lesssim  N^\epsilon \|a\|_{\ell^2}^p.
\end{equation*}
It is interesting to ask for what values of $p$ one can remove the $N^\epsilon$ factor. We shall prove that this is the case for $p=4$, provided that a weaker form of the following Cilleruelo--Granville conjecture from \cite{CG} holds.
\begin{conj}\label{conj:CG}
	For any $\gamma\in(0,1)$, every arc in $\mathbb{S}_+^1(N)$ of length $N^{\frac\gamma2}$ contains at most $C(\gamma)$ lattice points.
\end{conj}
The conjecture was proved to  be true in \cite{CC} for all $\gamma<\frac12$.
Our conditional result reads as follows.
\begin{thm}\label{thm:L4sph}
	Assume that Conjecture \ref{conj:CG} is true for some $\gamma>\frac12$. Let $\sigma$ be a measure on $\T^2$ such that \eqref{eq:decbeta} holds with some $\beta>0$
	Then the estimate
	\begin{equation}\label{eq:L4sph}
	\int_{\T^2} \left|\sum_{\n\in S_N}a_\n e(\n\cdot x)\right|^4\,d\sigma(x)\lesssim \|a\|_{\ell^2(S_N)}^4
	\end{equation}
	holds with the implicit constant independent of $N$.
\end{thm}

\begin{proof}Let us write $R=\sqrt{N}$. We have
	\begin{align}
	\nonumber\int_{\T^2} \left|\sum_{\n\in S_N}a_{\n}\, e(\n\cdot x)\right|^4d\sigma(x)&
	\lesssim\sum_{0\le j \le\log_2 R}2^{-j\beta}\sum\limits_{\substack{\n_1, \n_2,\n_3, \n_4\in S_N\\ |\n_1+\n_2-\n_3-\n_4|\simeq 2^j}}  |a_{\n_1}a_{\n_2}a_{\n_3}a_{\n_4}|+\|a\|_{\ell^2}^4
	\\\label{last!}
	&\lesssim \sum_{0\le j \le\log_2 R}2^{-j\beta}I_j\|a\|_{\ell^2}^4,
	\end{align}
	where
	$$I_j=\max_{(\n_1,\n_2)\in S_N\times S_N}|\{(\n_3,\n_4)\in S_N\times S_N:\;|\n_1+\n_2-\n_3-\n_4|\simeq 2^j\}.$$ We distinguish two regimes.
	If $2^{j}> R^{2\gamma-1}$, the trivial bound $I_j\lesssim_\epsilon N^\epsilon$ suffices. 	 Assume now that $2^{j}\le R^{2\gamma-1}$. We cover $\mathbb{S}_+^1(N)$ with arcs $\tau$ of length $\simeq (2^jR)^{1/2}$. The  Cordoba--Fefferman
geometric argument (see for example Section 3.2 in \cite{C}) shows that we can split the arcs $\tau$ into $O(1)$ many collections, such that
$$\dist(\tau_1+\tau_2,\tau_3+\tau_4)\gg 2^j$$
for each $\tau_1,\ldots, \tau_4$ in each collection, subject  only to the requirement that $\{\tau_1,\tau_2\}\not=\{\tau_3,\tau_4\}$. To see that this is indeed the case, note that after rescaling this is equivalent to the fact that the sums of two arcs of length $\delta=(2^j/R)^{1/2}$ on $\S_+^1(1)$ are separated by $\gg \delta^2$. Since each $\tau_i$ contains at most $C(\gamma)$ lattice points, it follows that $I_j=O(1)$. The contribution of these $j$ to \eqref{last!} is thus acceptable.

\end{proof}

\end{document}